\documentclass[a4paper,10pt]{amsart}

\usepackage[utf8]{inputenc}
\usepackage{amssymb,esint,amsmath,mathrsfs,amsthm,amsbsy}
\usepackage{enumerate}
\usepackage{color}

\usepackage[hypertexnames=false]{hyperref}
\hypersetup{urlcolor=blue, colorlinks=true}

\newtheorem{theorem}{Theorem}
\newtheorem{lemma}[theorem]{Lemma}
\newtheorem{proposition}[theorem]{Proposition}
\newtheorem{corollary}[theorem]{Corollary}
\newtheorem{definition}[theorem]{Definition}

\theoremstyle{definition}
\theoremstyle{remark}
\newtheorem{remark}[theorem]{Remark}

\newcommand{\Bigxi}{\Xi}
\newcommand{\A}{{\scriptscriptstyle {A}}}
\newcommand{\T}{{\scriptscriptstyle \mathrm{T}}}
\newcommand{\R}{\ensuremath{\mathbb{R}}}

\newcommand{\dif}{\ensuremath{\mathrm{d}}}
\newcommand{\INT}{\textup{int}}
\newcommand{\chain}{\ensuremath{\beta}}
\newcommand{\dd}{\,\mathrm{d}}
\newcommand{\diver}{\mathrm{div}}
\newcommand{\veps}{\varepsilon}
\newcommand{\e}{\textup{e}}

\DeclareMathOperator*{\wliminf}{lim \, inf \lower0.8ex\hbox{\textup{\scriptsize*}}}
\DeclareMathOperator*{\wlimsup}{lim \, sup \raise0.4ex\hbox{\textup{\scriptsize*}}}
\DeclareMathOperator*{\esssup}{ess \, sup}
\DeclareMathOperator*{\essinf}{ess \, inf}
\DeclareMathOperator{\sgn}{\textup{sign}}
\DeclareMathOperator{\tr}{\textup{tr}}

\DeclareMathOperator*{\co}{\mathit{o}}

\def\restriction#1#2{\mathchoice
              {\setbox1\hbox{${\displaystyle #1}_{\scriptstyle #2}$}
              \restrictionaux{#1}{#2}}
              {\setbox1\hbox{${\textstyle #1}_{\scriptstyle #2}$}
              \restrictionaux{#1}{#2}}
              {\setbox1\hbox{${\scriptstyle #1}_{\scriptscriptstyle #2}$}
              \restrictionaux{#1}{#2}}
              {\setbox1\hbox{${\scriptscriptstyle #1}_{\scriptscriptstyle #2}$}
              \restrictionaux{#1}{#2}}}
\def\restrictionaux#1#2{{#1\,\smash{\vrule height .8\ht1 depth .85\dp1}}_{\,#2}} 

\begin{document}

\title[Stability and duality for entropy and viscosity solutions]{Optimal stability results and nonlinear duality for
$L^\infty$ entropy and $L^1$ viscosity solutions}

\author[N.~Alibaud]{Natha\"el Alibaud}
\address[N.~Alibaud]{SUPMICROTECH-ENSMM\\
26 Chemin de l'Epitaphe\\ 25030 Besan\c{c}on ce\-dex\\ France\\
and\\
Universit\'e de Franche-Comt\'e\\ CNRS\\ LmB\\ F-25000 Besan\c{c}on\\ France}
\email{nathael.alibaud\@@{}ens2m.fr}
\urladdr{https://lmb.univ-fcomte.fr/Alibaud-Nathael}

\author[J. Endal]{J\o rgen Endal}
\address[J. Endal]{Department of Mathematical Sciences\\
Norwegian University of Science and Technology (NTNU)\\
N-7491 Trondheim, Norway} 
\email[]{jorgen.endal\@@{}ntnu.no}
\urladdr{http://folk.ntnu.no/jorgeen}

\author[E.~R.~Jakobsen]{Espen R. Jakobsen}
\address[E.~R.~Jakobsen]{Department of Mathematics\\
Norwegian University of Science and Technology (NTNU)\\
N-7491 Trondheim, Norway} \email[]{espen.jakobsen\@@{}ntnu.no}
\urladdr{http://folk.ntnu.no/erj}



\begin{abstract}
We give a new and rigorous duality relation between two central notions of weak solutions of nonlinear PDEs: entropy and viscosity solutions. 
It takes the form of the {\it nonlinear dual inequality:}
\begin{equation}\label{star}
\tag{$\star$}
\int |S_t u_0-S_t v_0| \varphi_0 \dd x\leq \int |u_0-v_0| G_t \varphi_0 \dd x, \quad \forall \varphi_0 \geq 0, \forall u_0, \forall v_0, 
\end{equation}
\renewcommand*{\theHequation}{notag.\theequation}
where $S_t$ is the entropy solution semigroup of the anisotropic degenerate parabolic equation 
\begin{equation*}
\partial_t u+\diver F(u)=\diver (A(u) D u),
\end{equation*}
and where we look for the smallest semigroup $G_t$ satisfying \eqref{star}. This amounts to  finding  an optimal weighted $L^1$ contraction  estimate  for $S_t$. 
Our main result is that $G_t$ is the viscosity solution semigroup of the Hamilton-Jacobi-Bellman equation
\begin{equation*}
\partial_t \varphi=\sup\nolimits _\xi \{F'(\xi) \cdot D \varphi+\tr(A(\xi) D^2\varphi)\}.
\end{equation*}  
Since weigthed $L^1$ contraction  results  are mainly used for possibly nonintegrable $L^\infty$ solutions  $u$, 
the natural spaces behind this duality are $L^\infty$ for $S_t$ and $L^1$ for $G_t$.  We therefore  develop a  corresponding  $L^1$ theory for viscosity solutions  $\varphi$. But  $L^1$ itself is too  large for  well-posedness,  and we rigorously identify the weakest $L^1$ type Banach setting  where  we can  have  it -- a  subspace of $L^1$ called $L^\infty_\INT$. 
A consequence of our results is a new domain of dependence like estimate for second order anisotropic degenerate parabolic PDEs. It is  given in  terms of a stochastic target problem and extends  in a natural way recent  results for first order hyperbolic PDEs 
by [N. Pogodaev, {\it J. Differ. Equ.,} 2018].
\end{abstract}

\maketitle

\tableofcontents

\section{Introduction}

In this paper we study two central notions of weak solutions of nonlinear PDEs and their interplay -- entropy solutions and viscosity solutions. Originally introduced for first order scalar conservation laws \cite{Kru70} and
Hamilton-Jacobi equations \cite{CrLi83}  respectively, both solution concepts have later been  extended to
second order PDEs \cite{Jen88,Ish89,Car99,ChPe03}. Conservation laws are divergence form equations arising in continuum physics \cite{Daf10}, while Hamilton-Jacobi equations are nondivergence form  equations from e.g. differential geometry and optimal control theory \cite{FlSo93,Bar94,BaCa-Do97}. The  well-posedness  of these equations  is an important  topic and  requires the entropy and viscosity solution theories in general.
The literature  is by now very large and  includes lots of applications. See \cite{FlSo93,Dib93,Bar94,BaCa-Do97,Vaz07,Daf10,CrIsLi92}  for the state-of-the-art.  

\medskip

 Here we develop a new connection between these solution concepts. It is already well-known that viscosity solutions are integrated entropy solutions in space dimension one \cite{Cas92,KaRi02,CoPe20}. Our connection is valid in any dimension and is expressed through weighted $L^1$ contraction results for entropy solutions:  The  optimal weight is the viscosity solution of a well-determined {\it dual equation}. Since $L^\infty$ is a natural space for such weighted estimates, we  need and do  develop an $L^1$ theory for viscosity solutions  of the dual equation.  Consequences are  a new domain of dependence like result for second order PDEs in terms of a stochastic target problem,  a new rigorous form of duality between $L^\infty$ entropy and $L^1$ viscosity solutions in terms of nonlinear semigroups, and a new characterization of viscosity supersolutions;  see \eqref{weighted-stochastic}, \eqref{intro-sg} and \eqref{nldi} respectively. 

The idea of using viscosity solutions to get estimates for entropy solutions was from \cite{EnJa14}.  The corresponding results were rather accurate but not optimal yet.  In this paper we prove optimal estimates for entropy solutions -- and -- that viscosity solutions are in fact needed to prove this optimality. This is exactly what leads to rigorous duality results. Also note that we consider nonlinear anisotropic diffusions as opposed to \cite{EnJa14}. For an early discussion and open questions about ``duality between nonlinear semigroups,'' see \cite[pp. 28--29]{BeWi94}. We also mention the recent papers \cite{CaGo-Capr, Por24} which study transport equations with linear diffusion through viscosity solutions of their dual equations. 

\medskip

To be more precise,  we consider the  following two Cauchy problems: For 
the anisotropic degenerate parabolic convection-diffusion equation
\begin{equation}\label{E}
\begin{aligned}
\partial_t u + \diver
F(u) = \diver \left(A(u) D u\right) 
& \qquad x \in \R^d, t>0,
\\
u(x,0)=u_0(x)  & \qquad x \in \mathbb{R}^d,
\end{aligned}
\end{equation} 
and  for  the
Hamilton-Jacobi-Bellman (HJB) 
equation 
\begin{subequations}\label{eq:dual}
\begin{align}
\partial_t \varphi=\sup\nolimits_{\xi\in \mathcal{E}} \left\{b(\xi) \cdot D \varphi+\tr \left(a(\xi)D^2\varphi\right) \right\} 
& \qquad x \in \R^d, t>0,\label{eq:dual-1}
\\
\varphi(x,0)=\varphi_0(x)   & \qquad x \in \R^d,
\end{align}
\end{subequations}
where 
``$D$,'' ``$D^2$'' and ``$\diver$'' 
respectively denote the
gradient, the Hessian and the divergence in $x$, and 
``$\tr$'' 
is the
trace. We assume that
\begin{equation}\label{fassumption}
\tag{{\rm H1}}
F\in
W_\textup{loc}^{1,\infty}(\mathbb{R},\mathbb{R}^d)
\quad\text{and}\quad A=\sigma^\A  \left(\sigma^{\A}\right)^{\T} \quad\mbox{for}\quad \sigma^{\A} \in L_\textup{loc}^\infty(\mathbb{R},\R^{d \times K}),
\end{equation}
\renewcommand*{\theHequation}{notag2.\theequation}
as well as
\begin{equation}\label{bassumption}
\tag{{\rm H2}}
\!\!\!\!\!\!\!\!\!\!\!\!\!\!\!\!\!\!\!\!\!\!\!\!\!\!\!\!\!\!\!\!\!\!\begin{cases}
\mathcal{E} \text{ is a nonempty set,}\\
b:\mathcal{E} \to\R^d \text{ a bounded function},\\
a=\sigma^{a} \left(\sigma^{a}\right)^{\T} \mbox{ for some bounded } \sigma^{a}:\mathcal{E} \to \R^{d \times K},
\end{cases}
\end{equation}
\renewcommand*{\theHequation}{notag3.\theequation}
where $K$ is the maximal rank of $A(u)$ and $a(\xi)$. The entropy solution theory for first order PDEs \cite{Kru70} was extended in \cite{Car99,ChPe03} to show well-posedness of \eqref{E} in $L^1 \cap L^\infty$ or $L^1$. Well-posedness in $L^\infty$ is less standard for second order PDEs, but results  exist  in
\cite{ChDi01,AnMa10,EnJa14,Pan20}; see \cite{Fripr} for anisotropic diffusions. 
Our main  objective  is to derive an optimal weighted $L^1$ contraction result for $L^\infty$ entropy solutions of \eqref{E}.  This then will require the  developement of a corresponding $L^1$ theory for  a dual equation of the form  \eqref{eq:dual}, a nonstandard generalization of classical viscosity solution theory   \cite{CrLi83,Jen88,Ish89,CrIsLi92,FlSo93,Bar94,BaCa-Do97}.


Contraction type estimates are quantitative continuous dependence results on the initial data. A simple example is the $L^1$ contraction 
principle \cite{Kru70,Car99,ChPe03}:
\begin{equation}
\label{intro-contraction}
\|(u-v)(t)\|_{L^1} \leq \|u_0-v_0\|_{L^1}.
\end{equation}
For possibly nonintegrable $L^\infty$ solutions,  we need weighted estimates. An important result is the finite speed of propagation property for first order PDEs \cite{Kru70}:
\begin{equation}\label{finite-propagation}
\int_{|x-x_0| <R} |u(x,t) -v(x,t)|  \dd x \leq \int_{|x-x_0|<R+Ct} |u_0(x)-v_0(x)| \dd x;
\end{equation}
see \cite{Pog18} for more precise estimates.
For second order PDEs, a standard example is given in 
\cite{BeCrPi84,ChDi01,WuZhYiLi01,Fripr}:
\begin{equation}\label{standard}
\int |u(x,t) -v(x,t)| \e^{-\sqrt{1+|x|^2}}  \dd x \leq \e^{C t} \int |u_0(x)-v_0(x)|  \e^{-\sqrt{1+|x|^2}} \dd x.
\end{equation} 
Note that \eqref{standard} does not imply \eqref{intro-contraction} and \eqref{finite-propagation}. A finer result that is closer to \eqref{finite-propagation} is given in \cite{EnJa14} but it still
does not imply \eqref{intro-contraction}, see \cite[Rem. 2.7(b)]{EnJa14}.
\medskip

We continue with a formal presentation of our main results. We first give
a very accurate  weighted  $L^1$ contraction estimate for \eqref{E}. We need to be precise about the dependence of the estimates in $u_0$ and $v_0$. Note that $C$ in \eqref{finite-propagation} and \eqref{standard} actually depends on $L^\infty$ bounds on these initial data. These bounds will determine $\mathcal{E}$ in the dual equation of the form \eqref{eq:dual}. For $m<M$, our new estimate for \eqref{E} is
\begin{equation}\label{weighted-intro}
\int |u(x,t)-v(x,t)| \varphi_0(x) \dd x \leq \int |u_0(x)-v_0(x)| \varphi(x,t) \dd x,
\end{equation}
where $\varphi_0\geq 0$ is arbitrary, the weight  $\varphi$ is the viscosity solution of \eqref{eq:dual} with 
\begin{equation}\label{hyp-referee}
b = F', \quad a = A, \quad \mathcal{E}=[m,M] \cap \big\{\text{Lebesgue points of $(F',A)$}\big\},
\end{equation}
and $u_0$ and $v_0$ are arbitrary with values in $[m,M]$.
For a precise statement, see Theorem \ref{thm:weighted}. Note that we also use another equivalent formulation of \eqref{eq:dual} in terms of $\esssup$, see \eqref{eq:dual-true}. The standard HJB form is, however, used especially for results specific to viscosity solutions.

Equation \eqref{eq:dual} is also related to stochastic control theory \cite{FlSo93}. 
If we assume $F'$ and $A$ are continuous so  $\mathcal{E}$ becomes compact, then the solution $\varphi$ of \eqref{eq:dual} is the value function of a stochastic target problem and \eqref{weighted-intro} can be rewritten as 
\begin{equation}
\label{weighted-stochastic}
\int_{U} |u(x,t) -v(x,t)| \dd x \leq \int |u_0(x)-v_0(x)| \sup_{\boldsymbol{\xi}_\cdot\in \Bigxi} \mathbb{P} \left(\boldsymbol{X}_t^x \in U\right) \dif x,
\end{equation}
 where $U \subseteq \R^d$ is arbitrary,  $\mathbb{P}$ is the probability, $\Bigxi$ is the set of $[m,M]$-valued processes, $\boldsymbol{B}_s$ a Brownian motion, and $\boldsymbol{X}_s^x$ an Ito process satisfying the stochastic differential equation (SDE) 
\begin{equation}\label{SDE}
\begin{aligned}
\dd \boldsymbol{X}_s^x=F'(\boldsymbol{\xi}_s) \dd s+\sqrt{2} \, \sigma^\A(\boldsymbol{\xi}_s) \dd \boldsymbol{B}_s & \qquad s >0,\\
 \boldsymbol{X}^x_{0}=x.
\end{aligned}
\end{equation}
For the precise statement, see Corollary \ref{cor:ws}. The control $\boldsymbol{\xi}_s$ is determined to maximize the probability for the controlled process $\boldsymbol{X}_s^x$  starting from $x$ at time $0$ to reach $U$ at
time $t$,  and Equation 
\eqref{eq:dual-1} with data \eqref{hyp-referee} is the dynamic programming
equation for this control problem. Interestingly \eqref{weighted-stochastic} can be interpreted as a domain of dependence estimate for \eqref{E}. 
Indeed if we consider the deterministic case $A \equiv 0$, then formally \eqref{SDE} becomes 
the characteristic equation of \eqref{E}, $\frac{\dd \boldsymbol{X}_s^x}{\dd s}=F'(u(\boldsymbol{X}_s^x,s))$, if we take
$\boldsymbol{\xi}_s=u(\boldsymbol{X}_s^x,s)$. In fact Estimate \eqref{weighted-stochastic} reduces to the domain of dependence estimate of \cite{Pog18} for scalar conservation laws, see Corollary \ref{cor:wd}.
This suggests that \eqref{weighted-stochastic} is a natural extension of such estimates to the degenerate parabolic equation \eqref{E}, where the second order term in \eqref{E} is taken into account via the Brownian part (the Ito integral) in \eqref{SDE}.

Note that \eqref{weighted-intro} and \eqref{weighted-stochastic} imply \eqref{intro-contraction}, \eqref{finite-propagation}, \eqref{standard}, the related results  in \cite{EnJa14,Pog18}, and as we will see, they are optimal in a rigorous sense. To discuss the optimality of \eqref{weighted-intro}, we fix $\varphi_0$ and try to identify the minimal $\varphi$ satisfying \eqref{weighted-intro} for any $u_0,v_0$. The key  result (Theorem \ref{thm:dual})
is   a characterization of viscosity supersolutions of \eqref{eq:dual} in terms of contraction estimates for \eqref{E}: 
\begin{quote}
{\it A nonnegative function $\varphi$ is a viscosity supersolution of \eqref{eq:dual-1} with data \eqref{hyp-referee} if and only if
\begin{equation}
\label{intro-sg}
\int |u(x,t)-v(x,t)| \varphi(x,s) \dd x \leq \int |u_0(x)-v_0(x)| \varphi(x,t+s) \dd x,
\end{equation}
for all $t,s \geq 0$ and $u_0,v_0$ with values in $[m,M]$ with associated entropy solutions $u,v$ of \eqref{E}.}
\end{quote}
Roughly speaking  this result  implies that if  we restrict  to weights satisfying a natural semigroup property, then the best weight in \eqref{weighted-intro} is the viscosity solution of \eqref{eq:dual}  since  by comparison solutions are always smaller than supersolutions. This  then leads  to our most original  result (Corollary \ref{cor:sg}):
\begin{quote}
{\it If $S_t$ and $G_t$ are the solution semigroups of
 \eqref{E} and \eqref{eq:dual}, with data \eqref{hyp-referee}, then $G_t$ is the smallest semigroup satisfying
\begin{equation}
\label{nldi}
\int |S_tu_0-S_tv_0| \varphi_0 \dd x \leq \int |u_0-v_0| G_t \varphi_0 \dd x,
\end{equation}}
for all $u_0,v_0$ with values in $[m,M]$ and nonnegative $\varphi_0$.
\end{quote}
We can interpret \eqref{nldi} as a {\em nonlinear dual inequality} and 
$G_t$ as a {\em dual  semigroup} of
$S_t$, because 
$G_t$ is entirely determined
by \eqref{nldi} and knowledge of $S_t$. 
The duality in the other direction is open
(Remark~\ref{noq}). Since
$S_t$ is taken on $L^\infty$ from the begining, it  remains  to properly define $G_t$ on $L^1$. 

Classical viscosity solution theory starting from \cite{CrLi83,Jen88,Ish89} and summarized in e.g. \cite{CrIsLi92,FlSo93,Bar94,BaCa-Do97}, typically considers bounded continuous $C_b$ solutions. 
For solutions in $L^1$ or $L^p$  (in space)  there are fewer results, see  e.g. \cite{CaCrKoSw96} for nondegenerate PDEs and
\cite{LiTa01,AlHoLBMo06,AlIm09,BaCaLeMo09,BoFoZi10,CaFoMo11,EnJa14} for various other PDEs. Here we show that \eqref{eq:dual} can be ill-posed in $L^1$  in general.   We then 
consider stronger norm topologies and identify the weakest one for which \eqref{eq:dual} is well-posed  in general:  It  is generated by the norm 
$$
\varphi_0 \mapsto \int \sup_{x + [-1,1]^d} |\varphi_0| \dd x
$$
which is the norm of  the space  $L^\infty_{\textup{int}}$ as defined in 
\cite{AlHoLBMo06,AlIm09}. 
 Since $L^\infty_{\textup{int}} \subset L^1 \cap L^\infty$, it follows that $C_b \cap L^\infty_\INT$ is a natural $L^1$ type Banach space for the  dual  equation \eqref{eq:dual} and its solution semigroup $G_t$ in  \eqref{nldi}; see Theorem \ref{thm:qualitative-HJB} and Corollary \ref{cor-nv}. 

\medskip

  Our  results on $L^1$ viscosity solutions are of independent  interest, see in particular Theorem \ref{thm:quantitative-HJB-unif}. Let us comment them  further.  The estimates of \cite{AlHoLBMo06} are not in $L^\infty_\INT$ but in its predual $L^1_{\rm unif} \not\subset L^1$,  while \cite{AlIm09} gives weighted $L^\infty_\INT$ estimates for unbounded solutions with linear diffusions.  In \cite{EnJa14} there are $L^1$ estimates for fully nonlinear degenerate PDEs with isotropic diffusions and exponentially decaying initial data.  Equation \eqref{eq:dual} is fully nonlinear, degenerate, possibly anisotropic, and we consider general $L^\infty_\INT$ data while identifying $L^\infty_\INT$ as the most natural $L^1$ viscosity solution setting.

\medskip

 The rest of this paper is organized as follows.  We recall basic facts in Section \ref{sec:preliminaires},  we  state our main results in Section \ref{main-results}, and prove them in Section \ref{sec:proofs}.  For completeness,  some results for  minimal  discontinuous viscosity solutions are proved in 
\ref{app:viscosity}, a complete proof of well-posedness for $L^\infty$ entropy solutions  is  given  in 
\ref{app:entropy}, and further comments on our duality results are postponed to 
\ref{app-comp-referee} and \ref{app-referee}.

\section{Preliminaries}\label{sec:preliminaires}

This section recalls  basic facts on  $C_b$  viscosity and $L^\infty$ entropy solutions; for proofs, see e.g.  \cite{CrIsLi92,FlSo93,Bar94,BaCa-Do97} and \cite{ChPe03,BeKa04,Daf10}  respectively. We  also define  the space $L^\infty_\INT$. 

\subsection{Notation} \label{sec:notation}
 Throughout $\R^+:=[0,\infty)$, balls and cubes of $\R^d$ with center $x$ and radius $r>0$ are $B_r(x):=\{y:|y-x| <r\}$ and $Q_r(x):=x+(-r,r)^d$, the symbol ``$\textup{co}$''  denotes the convex hull of sets, ``$(\textup{ess}) \, \textup{Im}$'' the (essential) image of (measurable) functions, ``$\textup{Sp}$'' the spectrum of matrices, and $\mathbf{1}_U$ the indicator function 
of a set $U$. 

We follow standard notation for function spaces, e.g. $C_c$ denotes continuous functions with compact support, $BLSC$ (resp. $BUSC$) bounded lower (resp. upper) semicontinuous functions, $L^p$ stands for Lebesgue spaces, etc. For two normed spaces $X \subseteq Y$, we say that $X$ is continuously embedded into $Y$ if the canonical injection is continuous, and the completion of $X$ is denoted by $\overline{X}^{\|\cdot\|_X} \subseteq Y$.  

Concerning operations on functions, ``$\ast$'' is the convolution which is mostly taken in $x \in \R^d$, and we use ``$\ast_{x,t}$'' if it is taken in $(x,t) \in \R^{d+1}$, etc. To regularize functions of $x$, we use convolution with an approximate unit $\rho_\nu$ of the form
\begin{equation}\label{space-approx}
\rho_\nu(x):=\frac{1}{\nu^d} \rho \left(\frac{x}{\nu}\right),
\end{equation}
where $0 \leq \rho \in C^\infty_c(\R^d)$ and $\int \rho=1$, while for functions of $t$, we convolve with 
\begin{equation}\label{time-approx}
\theta_\nu(t):=\frac{1}{\nu} \theta \left(\frac{t}{\nu}\right),
\end{equation}
where $0 \leq \theta \in C_c^\infty(
 (-\infty,0) 
)$ and  $\int \theta=1$. If needed, we extend functions of $t \in \R^+$ by zero to all $t\in\R$ to give a meaning to the convolution. For locally bounded everywhere defined $\varphi_0:\R^d \to \R$ or $\varphi:\R^d \times \R^+ \to \R$, we define the infconvolution \cite{CrIsLi92,FlSo93,Bar94,BaCa-Do97} of e.g.  $\varphi_0$ for all $x$ by
\begin{equation}\label{infconvol}
(\varphi_0)_\varepsilon(x):=\inf_{y \in \R^d} \left\{\varphi_0(y)+\frac{|x-y|^2}{2 \varepsilon^2}\right\}.
\end{equation}
Here the $\inf$ is pointwise and, to avoid confusion, we will use distinct notation for $\essinf$, etc. The upper $\varphi^\ast$ (lower $\varphi_\ast$) semicontinuous envelope of $\varphi$ is defined as 
$$
\varphi^\ast(x,t) := \limsup_{(y,s) \to (x,t)}\varphi(y,s)\quad \left(\varphi_\ast(x,t) := \liminf_{(y,s)\to (x,t)}\varphi(y,s)\right).
$$
For a family $\left(\varphi_\veps=\varphi_\veps(x,t) \right)_{\veps>0}$, the upper and lower relaxed limits as $\veps \to 0^+$ are, using standard notation \cite{CrIsLi92, Bar94, BaCa-Do97}, 
\begin{equation}\label{relaxed-referee}
\wlimsup \varphi_\veps(x,t):=\limsup_{\substack{(y,s)\to(x,t) \\ \veps \to 0^+}}\varphi_\veps(y,s) \quad \forall (x,t) \in \R^d \times \R^+,
\end{equation}
and $\wliminf \varphi_\veps:=-\wlimsup \left(-\varphi_\veps\right)$. As is customary, we use the same notation $\wlimsup$ and $\wliminf$ also when the limits are taken in another variable than $\veps\to0^+$, e.g. $R\to \infty$. We write $\lim_{\varepsilon \downarrow 0}\uparrow\varphi_\veps$ for the limit if $\varphi_\veps(x,t)\nearrow \sup_{\varepsilon > 0}\varphi_\veps (x,t)$ as $\veps\searrow0$. We use similar notation for $\varphi_0=\varphi_0(x)$ as e.g. $(\varphi_0)^\ast(x):=\limsup_{y \to x}\varphi_0(y)$, etc. 

As concerning stochastic processes, we fix
\begin{equation}
\label{assumption-stochastic}
\begin{cases}
\mbox{a complete filtered probability space $(\Omega,\mathcal{F},\mathcal{F}_t,\mathbb{P})$, and}\\
\mbox{a standard $d$-dimensional Brownian $\boldsymbol{B}_t$ on this filtration}.
\end{cases}
\end{equation}
The associated expectation w.r.t. $\mathbb{P}$ is denoted by $\mathbb{E}$.
We will assume possibly without mentioning that all stochastic processes in this paper are defined on this filtered probability space, and that whenever we need a Brownian motion, then we take the above Brownian motion. 

\subsubsection*{Less standard notation}
Following \cite{AlHoLBMo06,AlIm09},
\begin{equation}\label{def-linftyint-referee}
L^\infty_\INT(\R^d):= \left\{\varphi_0 \in L^1_{\rm loc}(\R^d)\,:\, \|\varphi_0\|_{L^\infty_\INT}<\infty \right\},
\end{equation}
where $\|\varphi_0\|_{L^\infty_\INT} := \int \esssup_{\overline{Q}_1(x)} |\varphi_0| \dd x$. For the pointwise sup, we use 
$\|\varphi_0\|_\INT:=\int \sup_{\overline{Q}_1(x)} |\varphi_0| \dd x$. Note that $\|\varphi_0\|_\INT=\|\varphi_0\|_{L^\infty_\INT}$ if $\varphi_0$ is continuous. For more details about $L^\infty_\INT$, see Section \ref{sec-linftyint-int-referee}. 

For any $\varphi \in BLSC(\R \times \R^+)$, we associate a particular envelope defined as
\begin{equation}
\label{precise-representative}
\varphi_\#(x,t):= \liminf_{\substack{r \to 0^+ \\ y \to x }} \frac{1}{\textup{meas}(B_r(y))} \int_{B_r(y)} \varphi(z,t) \dd z.
\end{equation} 
This envelope will appear naturally in Theorem \ref{thm:dual} and more properties will be given in Section \ref{sec:proofs-bis}.

\subsection{Viscosity solutions of \eqref{eq:dual}}

We begin by introducing the correct notion of solutions for HJB equations \cite{CrIsLi92,FlSo93,Bar94,BaCa-Do97}.

\begin{definition}[Viscosity solutions]\label{def:viscsoln}
Assume \eqref{bassumption} and $\varphi_0:\R^d \to \R$ is bounded. 
\begin{enumerate}[\rm (a)]
\item A locally bounded function $\varphi:\R^d \times \R^+ \to\R$ is a viscosity subsolution (resp. supersolution) of \eqref{eq:dual} if
\smallskip
\begin{enumerate}[\rm (i)]
\item \label{rem-vse} for every $\phi \in C^\infty(\R^d\times \R^+)$ and local maximum $(x,t) \in \R^d \times (0,\infty)$ of $\varphi^\ast-\phi$ (resp. mininimum of $\varphi_\ast-\phi$),  
\begin{equation*}
\partial_t\phi(x,t)\leq \sup\nolimits_{\mathcal{E}} \left\{b \cdot D \phi(x,t)+\tr \left(a D^2\phi(x,t)\right) \right\}   \quad \mbox{(resp. $\geq$)},
\end{equation*}
\item\label{ref:id} and for every $x \in \R^d$,
$$
\varphi^\ast(x,0)\leq (\varphi_0)^\ast(x)  \quad \text{(resp. $\varphi_\ast(x,0) \geq (\varphi_0)_\ast(x)$)}.
$$
\end{enumerate} 
\item A function $\varphi$ is a viscosity solution if it is both a sub and supersolution. 
\end{enumerate}
\end{definition}

\begin{remark}
We say that $\varphi$ is a viscosity subsolution (resp. supersolution)  of \eqref{eq:dual-1} if \eqref{rem-vse} holds.
\end{remark}

We recall the well-known comparison and the well-posedness for \eqref{eq:dual} \cite{CrIsLi92,FlSo93}.

\begin{theorem}[Comparison principle]
\label{propofvissoln}
Assume \eqref{bassumption}. If $\varphi$ and $\psi$ are bounded sub and supersolutions of \eqref{eq:dual-1}, and 
\begin{equation*}
\varphi^\ast(x,0) \leq \psi_\ast(x,0) \quad \forall x \in \R^d,  
\end{equation*}
then $\varphi^\ast \leq \psi_\ast$ on $\R^d \times \R^+$. 
\end{theorem} 

\begin{theorem}[Existence and uniqueness]
\label{exisuniqvissoln}
Assume $\eqref{bassumption}$ and $\varphi_0\in C_b(\R^d)$. Then there exists a unique viscosity solution $\varphi \in C_b(\R^d \times \R^+)$ of \eqref{eq:dual}. 
\end{theorem}

\begin{remark}\label{rem:max}
 By the comparison principle, $\inf \varphi_0 \leq \varphi \leq \sup \varphi_0$ and we have the following contraction property:
$
\|\varphi-\psi\|_\infty \leq \|\varphi_0-\psi_0\|_\infty
$
for every  pair of  solutions $\varphi$ and $\psi$ with initial data $\varphi_0$ 
and $\psi_0$. 
\end{remark}

We may take $\varphi_0$ to be discontinuous as in \eqref{weighted-stochastic}. In that case, we lose  uniqueness and we have to work with minimal and maximal solutions \cite{CrLiSo89,BaSoSo93,GiSa01} (see also \cite{BaCa-Do97} for bilateral solutions). 
 For our considerations, we only need minimal solutions.

\begin{theorem}[Minimal solutions]\label{thm:propextremalsoln1}
Assume $\eqref{bassumption}$ and $\varphi_0:\R^d \to \R$ bounded. Then there exists a minimal viscosity solution $
\underline{\varphi} \in BLSC(\R^d \times \R^+) 
$ of \eqref{eq:dual},
in the sense that
$
\underline{\varphi} \leq \varphi$ for any bounded viscosity solution $\varphi$ of \eqref{eq:dual}.
Moreover 
$
\underline{\varphi}(x,t=0)=(\varphi_0)_\ast(x)
$
for any $x \in \R^d$.
\end{theorem}

Note that $\underline{\varphi}$ is unique by definition. Actually, it is more precisely the minimal supersolution.

\begin{proposition}\label{prop:extremal-reformulation}
Assume $\eqref{bassumption}$ and $\varphi_0:\R^d \to \R$ is bounded. Then any bounded supersolution $\varphi$ of \eqref{eq:dual} is such that $
\underline{\varphi} \leq \varphi_\ast$. 
\end{proposition}

\begin{remark}\label{comp-extremal}
In particular, we have the following comparison principle: $\underline{\varphi} \leq \underline{\psi}$ for any bounded $\varphi_0 \leq \psi_0$.
\end{remark}

For completeness, the proofs of Theorems \ref{thm:propextremalsoln1} and  Proposition  \ref{prop:extremal-reformulation} are given in 
\ref{app:min-max} because  \cite{CrLiSo89,BaSoSo93,GiSa01,BaCa-Do97} consider slightly different problems. Let us continue with representation formulas for the solution $\underline{\varphi}$ from control theory \cite{FlSo93,BaCa-Do97,GoSe10,GoSe11}.

\begin{proposition}[First order]\label{representation-deterministic}
Assume \eqref{bassumption}, $a \equiv 0$, and $\varphi_0:\R^d \to \R$ bounded. Then the minimal viscosity solution of \eqref{eq:dual} is given by
\begin{equation*}
\underline{\varphi}(x,t)=\sup_{x+t\mathcal{C}} (\varphi_0)_\ast \quad \forall (x,t) \in \R^d \times \R^+,
\end{equation*}
where 
$
\mathcal{C}= \overline{\textup{co} \left\{\textup{Im} (b)\right\}}.
$
\end{proposition}

In the second order case, we need a probabilistic framework.

\begin{proposition}[Second order]\label{rem-dyn}
Assume \eqref{bassumption}, \eqref{assumption-stochastic}, and  
\begin{equation}\label{ass:regular}
\mbox{the set $\mathcal{E}$ is compact and the functions $b(\cdot)$ and $\sigma^a(\cdot)$ are continuous.}
\end{equation} 
Then the minimal viscosity solution of \eqref{eq:dual} is given by 
\begin{equation*}
\underline{\varphi}(x,t)=\sup_{\boldsymbol{\xi}_\cdot\in\Bigxi} \mathbb{E} \left\{(\varphi_0)_\ast (\boldsymbol{X}_t^x)\right\},
\end{equation*}
where $\Bigxi$ is the set of progressively measurable $\mathcal{E}$-valued processes  and $\boldsymbol{X}_s^x$ an Ito process satisfying the SDE
$$
\begin{cases}
\dd \boldsymbol{X}_s^x=b(\boldsymbol{\xi}_s) \dd s+\sqrt{2} \, \sigma^a(\boldsymbol{\xi}_s) \dd \boldsymbol{B}_s, \quad s > 0,\\ 
\boldsymbol{X}_{s=0}^x=x.
\end{cases}
$$
\end{proposition}

These results are standard for continuous viscosity solutions \cite{FlSo93,BaCa-Do97}, see also \cite{BaCa-Do97,GoSe10,GoSe11} for maximal solutions. For minimal solutions, we did not find any reference so we provide the proofs in 
\ref{app:representation}. 

\subsection{Entropy solutions of \eqref{E}}

Well-posedness  of \eqref{E} in $L^\infty$ is  essentially  established  in \cite{Fripr} for smooth fluxes,  see \cite{ChPe03,BeKa04} for previous results in $L^\infty \cap L^1$ or $L^1$.  Let us now recall these  results in the form  needed here  and provide  complementary  proofs in 
\ref{app:entropy} for completeness.

\begin{definition}[Entropy-entropy flux triple]
We say that $(\eta,q,r)$ is an entropy-entropy flux triple if $\eta \in C^2(\R)$ is convex, $q'=\eta' F'$ and $r'=\eta' A$.
\end{definition}

Given $\beta \in C(\R)$, we also need the notation
\begin{equation*}
\zeta_{ik}(u):=\int_0^u\sigma^{\scriptscriptstyle A}_{ik}(\xi)\dd \xi \quad \mbox{and} \quad \zeta_{ik}^\chain(u):=\int_0^u \sigma^{\scriptscriptstyle A}_{ik}(\xi)\chain(\xi)\dd \xi.
\end{equation*}

\begin{definition}[Entropy solutions]\label{defentropy}
Assume \eqref{fassumption} and $u_0\in L^\infty(\R^d)$. 
A function
$u \in L^\infty(\R^d \times \R^+)\cap C(\R^+;L_{\textup{loc}}^1(\R^d))$ is an entropy solution of \eqref{E} if 
\begin{enumerate}[{\rm (a)}]
\item\label{item:energy}
$
\sum_{i=1}^d\partial_{x_i}\zeta_{ik}(u)\in L_{\textup{loc}}^2(\R^d \times \R^+)
$
for any $k=1,\ldots,K$,
\smallskip
\item\label{item:chain} for any $k=1,\ldots,K$ and any $\chain\in C(\R)$
$$
\sum_{i=1}^d\partial_{x_i}\zeta_{ik}^\chain(u)=\chain(u)\sum_{i=1}^d\partial_{x_i}\zeta_{ik}(u) \in L_{\textup{loc}}^2(\R^d \times \R^+),
$$
\item\label{item:entropy} and for all entropy-entropy flux triples $(\eta,q,r)$ and $0 \leq \phi\in C_c^\infty(\R^d \times \R^+)$,
\begin{equation*}
\begin{split}
&\iint_{\R^d \times \R^+} \left(\eta(u)\partial_t\phi+\sum_{i=1}^d q_i(u)\partial_{x_i}\phi+\sum_{i,j=1}^{d}r_{ij}(u)\partial_{x_ix_j}^2\phi\right)\dif x\dd t\\
&+\int_{\R^d} \eta(u_0(x)) \phi(x,0) \dd x \geq\iint_{\R^d \times \R^+}\eta''(u)\sum_{k=1}^K\left(\sum_{i=1}^d\partial_{x_i}\zeta_{ik}(u)\right)^2\phi\dd x \dd t.
\end{split}
\end{equation*}
\end{enumerate}
\end{definition}

\begin{theorem}[Existence and uniqueness]
\label{thm:existenceuniquenessentropysoln}
Assume \eqref{fassumption} and $u_0 \in L^\infty(\R^d)$. Then there exists a unique entropy solution $u \in L^\infty(\R^d \times \R^+)\cap C(\R^+;L_{\textup{loc}}^1(\R^d))$ of \eqref{E}. 
\end{theorem}

See \cite[Theorem 1.1]{Fripr} or \ref{app:entropy} for the proof.

\begin{remark}
\begin{enumerate}[{\rm (a)}]
\item 
In the $L^1$ settings  
of \cite{ChPe03,BeKa04}, the following contraction principle holds: For solutions $u$ and $v$ of \eqref{E} with initial data $u_0$ and $v_0$,
\begin{equation*}
\|u(\cdot,t)-v(\cdot,t)\|_{L^1} \leq \|u_0-v_0\|_{L^1} \quad \forall t \geq 0.
\end{equation*}
\item In the $L^\infty$ setting 
of \cite{Fripr}, uniqueness is based on the weighted $L^1$ contraction principle \eqref{standard}, see also Lemma \ref{standarduniqueness} in 
%
\ref{app:entropy}.
\item In all 
cases,
we have comparison and maximum principles as stated in
Lemma \ref{lem:max-comp} in 
%
\ref{app:entropy}.
\end{enumerate}
\end{remark}
In $L^\infty$, uniqueness is based on  a doubling of variables arguments
 developed in \cite{Kru70,Car99,BeKa04}. This argument leads  to  \eqref{dualeq} below, and this inequality will be the starting point of our  analysis. 

\begin{lemma}[Kato inequality]
\label{dualequation}
Assume \eqref{fassumption} and $u,v$ are entropy solutions of \eqref{E} with initial
data $u_0, v_0\in L^\infty(\R^d)$. Then for all $T \geq 0$ and nonnegative test functions $\phi \in C_c^{\infty}(\R^d \times [0,T])$,
\begin{multline}
\int_{\R^d} |u-v|(x,T)\phi(x,T)\dd x \leq \int_{\R^d}|u_0-v_0|(x)\phi(x,0)\dd x\\
+\iint_{\R^d \times (0,T)} \left(|u-v| \partial_t \phi+\sum_{i=1}^d q_i(u,v)\partial_{x_i}\phi+\sum_{i,j=1}^d r_{ij}(u,v)\partial_{x_ix_j}^2\phi\right) \dif x \dd t,
\label{dualeq}
\end{multline}
where
$$
q_i(u,v) =\sgn(u-v) \int_v^u F_i'(\xi) \dd \xi, \quad
r_{ij}(u,v) =\sgn(u-v)\int_v^u A_{ij}(\xi)\dd \xi.
$$
\end{lemma}

See 
\ref{app:entropy} for precise references to the computations in \cite{BeKa04} on how to show this lemma in our setting. 

\subsection{The function space $L^\infty_\INT$}\label{sec-linftyint-int-referee}

Let us now give some basic properties on the space which was defined in \eqref{def-linftyint-referee}.

\begin{theorem}\label{thm:uniformly-integrable}
 The space  $L^\infty_\INT(\R^d)$ is a Banach space, and it  is continuously embedded into $L^1 \cap L^\infty(\R^d)$.
\end{theorem}

 See \cite{AlHoLBMo06,AlIm09} for the proof and choice of  the above notation.  We also need the following result:

\begin{lemma}\label{lem:criterion}
For any $r > 0$ and $\varepsilon \geq 0$, there is a constant $C_{r,\varepsilon} \geq 0$ such that 
\begin{equation*}
\int \sup_{\overline{Q}_{r+\varepsilon}(x)} |\varphi_0| \dd x \leq C_{r,\varepsilon} \int \sup_{\overline{Q}_r(x)} |\varphi_0| \dd x \quad \forall \varphi_0:\R^d \to \R.
\end{equation*}
\end{lemma}

\begin{remark}
This result will be used with the pointwise sup for discontinuous $\varphi_0$, typically lower or upper semicontinuous.
\end{remark}

The proof can be found in  \cite{AlHoLBMo06,AlIm09}, see e.g.  \cite[Lemma 2.5.1]{AlIm09}. 

\section{Main results}
\label{main-results}

In this section we precisely state our results: the weighted $L^1$ contraction estimate for \eqref{E} in Section \ref{sec:weighted}, the optimality of the weight in Section \ref{sec:duality}, and the interpretation in terms of dual nonlinear semigroup in Section \ref{sec:unif-int}. Section \ref{sec:unif-int} contains the $L^1$ theory for \eqref{eq:dual}, and the long proofs are postponed
to Section \ref{sec:proofs}. 

\subsection{Weighted $L^1$ contraction for entropy solutions}
\label{sec:weighted}

The weight $\varphi$ of our new estimate for \eqref{E} is the viscosity solution of \eqref{eq:dual} with data \eqref{hyp-referee}, a problem which we rewrite in the more convenient form\footnote{Viscosity solutions are understood as in Definition \ref{def:viscsoln} via Problem \eqref{eq:dual} with data \eqref{hyp-referee}. But we let the reader check that we can equivalently redefine this notion via \eqref{eq:dual-true}. More precisely $\varphi$ is a viscosity supersolution of \eqref{eq:dual} with data \eqref{hyp-referee} if and only if for every $\phi \in C^\infty$ and local max $(x,t)$ of $\varphi^\ast-\phi$, $\partial_t \phi(x,t) \leq \esssup_{m \leq \xi \leq M} \left\{F'(\xi) \cdot D \phi(x,t)+\tr \left(A(\xi) D^2\phi(x,t)\right) \right\}$, etc.} 
\begin{subequations}\label{eq:dual-true}
\begin{align}
\partial_t \varphi=\esssup_{m \leq \xi \leq M} \left\{F'(\xi) \cdot D \varphi+\tr \left(A(\xi)D^2\varphi\right) \right\}
& \qquad x \in \R^d, t>0,\label{eq:dual-true-1}
\\
\varphi(x,0)=\varphi_0(x)   & \qquad x \in \R^d,
\end{align}
\end{subequations}
for given $m < M$ and  $\varphi_0$. 

\begin{theorem}[Weighted $L^1$ contraction]
\label{thm:weighted}
 Assume \eqref{fassumption}, $m<M$, $u_0=u_0(x)$ and $v_0=v_0(x)$ are measurable with values in $[m,M]$, and $0 \leq \varphi_0 \in BLSC(\R^d)$. Then the corresponding entropy solutions $u$ and $v$ of \eqref{E} and minimal viscosity solution $\underline{\varphi}$ of \eqref{eq:dual-true} satisfy 
\begin{equation}\label{nonlinear-dual}
\int_{\R^d} |u-v|(x,t) \varphi_0(x) \dd x \leq \int_{\R^d} |u_0-v_0|(x) \underline{\varphi}(x,t)\dd x \quad \forall t\geq 0.
\end{equation}
\end{theorem}

\begin{remark}\label{rem:weighted}
\begin{enumerate}[{\rm (a)}]
\item
The right-hand side of \eqref{nonlinear-dual} can be infinite. To get finite integrals, it suffices to take $u_0-v_0\in L^1$. We shall see later that another sufficient condition is that $\varphi_0 \in L^\infty_\INT$, since $\underline{\varphi}$ will then be $L^1$ in space by Theorem \ref{thm:quantitative-HJB-unif}.
\item\label{weighted-supersol} The same result holds
when $\underline{\varphi}$ is replaced by any measurable supersolution of  \eqref{eq:dual-true},
since it is greater than $\underline{\varphi}$. 
\item We also point the interested reader to Lemma \ref{lem:L1}. There we prove that a consequence of the above result is that $u_0-v_0\in L^1$ implies $u-v \in C(\R^+;L^1(\R^d))$.
\end{enumerate}
\end{remark}

From control theory there exist representation formulas for $\varphi$ in the first and second order cases, see Propositions \ref{representation-deterministic} and \ref{rem-dyn}. Combining the above result with these representation formulas give us very precise domain of dependence results. In the first order case, we recover the precise results of \cite{Pog18}, while in the second order case the result is new.

\begin{corollary}[First order equations]
\label{cor:wd}
Assume \eqref{fassumption} with $A \equiv 0$,
$m<M$, $u_0$ and $v_0$ are measurable functions with values in
$[m,M]$, and $u$ and $v$ are entropy solutions of \eqref{E} with initial data $u_0$ and $v_0$. Then
\begin{equation*}
\int_{B} |u-v|(x,t) \dd x \leq \int_{B-t \mathcal{C}} |u_0-v_0|(x) \dd x
\end{equation*}
for
any Borel set $B \subseteq \R^d$ and $t \geq 0$, where
$$
\mathcal{C}=\overline{\textup{co} \left\{\textup{ess} \, \textup{Im} \left(\restriction{(F')}{[m,M]} \right) \right\}}.
$$
\end{corollary}

\begin{proof}
Let $U \supseteq B$ be an open set and take $\varphi_0=\mathbf{1}_U$. By  Proposition \ref{representation-deterministic}, the minimal solution of \eqref{eq:dual-true} is $\underline{\varphi}(x,t)=\mathbf{1}_{U-t \mathcal{C}}(x)$. Apply then Theorem \ref{thm:weighted} and take the infimum over all open $U \supseteq B$. 
\end{proof}

\begin{corollary}[Second order equations]\label{cor:ws}
Assume \eqref{fassumption}, \eqref{assumption-stochastic}, 
$F'(\cdot)$ and $\sigma^\A(\cdot)$ continuous, $m<M$, $u_0$ and $v_0$ in $L^\infty(\R^d,[m,M])$, and $u$ and $v$ entropy solutions of \eqref{E} with $u_0$ and $v_0$ as initial data. Then for any open $U \subseteq \R^d$ and $t \geq 0$, 
\begin{equation*}
\int_{U} |u-v|(x,t) \dd x \leq \int_{\R^d} |u_0-v_0|(x) \sup_{\boldsymbol{\xi}_\cdot\in\Bigxi} \mathbb{P} \left(\boldsymbol{X}_t^x \in U\right) \dif x,
\end{equation*}
where $\Bigxi$ is the set of progressively measurable $[m,M]$-valued processes and $\boldsymbol{X}_s^x$  is an Ito process satisfying the SDE \eqref{SDE}.
\end{corollary}

\begin{proof}
Take $\varphi_0=\mathbf{1}_U$ and apply Proposition \ref{rem-dyn} to compute $\underline{\varphi}$ in Theorem \ref{thm:weighted}.
\end{proof}

The proof of Theorem \ref{thm:weighted} is given in Section \ref{sec:contraction}.

\subsection{Optimality of the weight}\label{sec:duality}

Let us now discuss the optimality of the weight $\underline{\varphi}$ in a weighted $L^1$ contraction estimate for \eqref{E} such as \eqref{nonlinear-dual}. The first step is a reformulation of the definition of viscosity supersolutions of \eqref{eq:dual-true-1} in terms of weights in $L^1$ contraction estimates for \eqref{E}. 

\begin{theorem}[Weights and supersolutions]\label{thm:dual} Assume \eqref{fassumption}, $m<M$, and $0 \leq \varphi \in BLSC(\R^d \times \R^+)$. Then the  statements  below are equivalent.
\begin{enumerate}[{\rm (I)}]
\item\label{nonlinear-dual-optimal} For any measurable functions $u_0$ and $v_0$ with values in $[m,M]$ and entropy solutions $u$ and $v$ of \eqref{E} with initial data $u_0$ and $v_0$,
\begin{equation*}
\int_{\R^d} |u-v|(x,t) \varphi(x,s) \dd x \leq \int_{\R^d} |u_0-v_0|(x) \varphi(x,t+s)\dd x \quad \forall t,s \geq 0.
\end{equation*}
\item\label{item:dual} The function $\varphi_\#$ (cf. \eqref{precise-representative}) is a viscosity supersolution of \eqref{eq:dual-true-1}. 
\end{enumerate}
\end{theorem}
 
\begin{remark}
\begin{enumerate}[\rm (a)]
\item We will see in Lemma \ref{lem:limit-reg}\eqref{item:star} that $\varphi_\#(\cdot,t)=\varphi(\cdot,t)$
a.e. in $\R^d$, for any $t$. Hence $\varphi_\#$ satisfies \eqref{nonlinear-dual-optimal} if and only if $\varphi$ does. 
\item For a fixed $t$, the classical precise representative \cite{EvGa15,Pon16} of $\varphi(\cdot,t)$, is defined over Lebesgue points (in space) as
$$
\hat{\varphi}(x,t):=\lim_{r \to 0^+} \frac{1}{\textup{meas}(B_r(x))} \int_{B_r(x)} \varphi(y,t) \dd y.
$$
Assigning the value $\sup \varphi$ at all other points, and taking the lower semicontinuous envelope (in $x$), will exactly give $\varphi_\#(\cdot,t)$. 
\item Although $\varphi \in BLSC$ makes sense everywhere, we need to consider another precise representative in $x$ for the viscosity inequalities to hold. This is because these inequalities are pointwise while \eqref{nonlinear-dual-optimal} does not depend on the choice of such representatives. If e.g. modifying $\varphi$ only at some $(x_0,t_0)$ such that
$$
\varphi(x_0,t_0) < \liminf_{(x_0,t_0) \neq (x,t) \to (x_0,t_0)} \varphi(x,t),
$$ 
we would preserve \eqref{nonlinear-dual-optimal} while losing the viscosity inequalities.\footnote{Indeed $\varphi-\phi$ would achieve a local min in $(x_0,t_0)$, for all $\phi \in C^\infty$.}
\item We do not need to change  the precise representative in $t$, roughly speaking because we consider $BLSC$ weights satisfying \eqref{nonlinear-dual-optimal} for all times.
\item For simplicity, we restrict to $BLSC$ weights since this regularity is shared by $\underline{\varphi}$ from Theorem \ref{thm:weighted} and most of the weights from the literature. But we have a similar result for merely measurable weights in $(x,t)$; see 
\ref{app-comp-referee} for completeness.
\end{enumerate}
\end{remark}

We will therefore roughly speaking deduce from the comparison principle that our weight is optimal in the class of weights 
\begin{equation*}
\mathscr{W}_{m,M,\varphi_0}:=
\left\{0 \leq \varphi \in BLSC(\R^d \times \R^+) \text{ satisfying \eqref{nonlinear-dual-optimal} and $\varphi(t=0) \geq \varphi_0$} \right\}.
\end{equation*}

\begin{corollary}[Optimality of the weight]\label{cor:opt} Assume \eqref{fassumption}, $m<M$, and $0 \leq \varphi_0 \in BLSC(\R^d)$. 
Then the weight $\underline{\varphi}$ from Theorem \ref{thm:weighted} belongs to the class $\mathscr{W}_{m,M,\varphi_0}$ and satisfies
\begin{equation*}
(\underline{\varphi})_\#(x,t) = \inf \left\{\varphi_\#(x,t): \varphi \in \mathscr{W}_{m,M,\varphi_0} \right\} \quad \forall (x,t) \in \R^d \times \R^+.
\end{equation*}
\end{corollary}

\begin{remark}
\begin{enumerate}[\rm (a)]
\item 
Property \eqref{nonlinear-dual-optimal} is stronger than \eqref{nonlinear-dual} since it holds for any $s \geq 0$. This may be interpreted as a certain semigroup property.
\item Property \eqref{nonlinear-dual-optimal} is satisfied by most of the weights from the literature, as e.g. for
$$
\varphi \equiv 1, \quad \varphi(x,t)=\mathbf{1}_{|x-x_0|<R+C t} \quad \mbox{and} \quad \varphi(x,t)=\e^{C t} \e^{-\sqrt{1+|x|^2}},
$$ 
in respectively \eqref{intro-contraction}, \eqref{finite-propagation} and \eqref{standard}; see also the stability results from \cite{Kru70,BeCrPi84,ChDi01,WuZhYiLi01,ChPe03,EnJa14,Fripr,Pog18}.
\end{enumerate}
\end{remark}

The proofs of Theorem \ref{thm:dual} and Corollary \ref{cor:opt} are given in Section \ref{sec:proofs-bis}. 

\subsection{$L^\infty_\INT$, semigroup formulation, and a new form of duality} 
\label{sec:unif-int}
We now interpret our results in terms of semigroups. This will reflect some form of duality for nonlinear semigroups, which will reduce to standard duality in the linear case. We first need to make the functional framework precise. Recall that $L^1$ might seem natural for the dual semigroup which will correspond to the weights in \eqref{nonlinear-dual}, but it is too weak for HJB equations and we will precisely explain why $L^\infty_\INT \subset L^1$ is a better and very natural setting. This preliminary study has also its own interest in viscosity solution theory, and is written for HJB equations in the usual form \eqref{eq:dual}. 

\subsubsection*{Preliminaries: $C_b\cap L^\infty_\INT$ as a natural $L^1$ setting for \eqref{eq:dual}}

We first explain why the pure $L^1$ setting is too weak to develop a general well-posedness theory for \eqref{eq:dual}. Consider a solution of the eikonal type equation\footnote{
Equation \eqref{eikonal} is of the form \eqref{eq:dual} with $\mathcal{E}=\overline{Q}_1(0)$, $b(\xi)=\xi$, and $a\equiv0$.} 
\begin{equation}
\label{eikonal}
\partial_t \varphi=\sum_{i=1}^d |\partial_{x_i} \varphi|.
\end{equation}
Under which condition is it integrable?

\begin{proposition}[Necesssary and sufficient integrability condition]
Let $\varphi$ be the viscosity solution of \eqref{eikonal}  with  initial data  $\varphi_0\in C_b (\R^d)$.
We then have
$$
\left[\varphi(\cdot,t) \in L^1(\R^d) \quad \forall t \geq 0\right] \quad\Longleftrightarrow\quad  \left[\varphi_0^- \in L^1(\R^d) \mbox{ and } \varphi_0^+ \in L^\infty_\INT(\R^d) \right].
$$
\end{proposition}

\begin{proof}
 Since $
\varphi(x,t)=\sup_{\overline{Q}_{t}(x)} \varphi_0$ by Proposition \ref{representation-deterministic}, we conclude by  Lemma \ref{lem:criterion}.
\end{proof}

We continue by showing that the  $L^1$ topology is too weak to get the continuous dependence on  the initial data, even for solutions which remain integrable.
  
\begin{proposition}[Failure of the $L^1$ continuous dependence]
 For all  $n \geq 1$, let $\varphi_0^n(x):=(1-n|x|)^+$, and $\varphi_n$ be 
the solution of \eqref{eikonal}  with  initial data  $\varphi_0^n$. Then $\varphi_0^n \in C_b \cap L^\infty_\INT(\R^d)$ and
$$
\lim_{n \to \infty }\varphi_0^n= 0 \quad \mbox{in } L^1(\R^d),
$$
but
$$
\lim_{n \to \infty} \varphi_n(\cdot,t)=\mathbf{1}_{\overline{Q}_t}(\cdot)\neq 0 \quad \mbox{in } L^1(\R^d), \quad \forall t>0.
$$
\end{proposition}

\begin{proof}
Use again that $\varphi_n(x,t)=\sup_{\overline{Q}_{t}(x)} \varphi_0^n$.
\end{proof}

Interestingly  a similar  analysis  works also for purely diffusive HJB 
equations. Consider e.g. an equation in one space dimension\footnote{Equation \eqref{example-diffusion} is of the form \eqref{eq:dual} with $\mathcal{E}=[0,1]$, $b\equiv0$, and $a(\xi)=\xi$.}
\begin{equation}\label{example-diffusion}
\partial_t \varphi=\left(\partial_{xx}^2 \varphi\right)^+.
\end{equation} 
To have $L^1$ solutions, we need again that  $\varphi_0^+ \in L^\infty_\INT$.

\begin{proposition}[$L^\infty_\INT$ and nonlinear diffusions]\label{natural-second}
Let $\varphi_0 \in C_b(\R)$ be nonnegative and 
$\varphi$ be the viscosity solution of \eqref{example-diffusion} with $\varphi_0$ as initial data. 
Then, 
$$
\left[\varphi(\cdot,t) \in L^1(\R) \quad \forall t \geq 0\right] \quad\Longleftrightarrow\quad  \varphi_0 \in L^\infty_\INT(\R).
$$
\end{proposition}

See Section \ref{ref-title-precise-instability} for the proof. We  now  use  the lack of a fundamental solution  of  \eqref{example-diffusion}  to show that there is no continuous dependence on the initial data in $L^1$.    

\begin{proposition}[Blow-up everywhere]
\label{lem:blow-up}
For all $n\geq 1$, let $\varphi_n$ be the viscosity solution of \eqref{example-diffusion} with an approximate delta-function as initial data: 
\begin{equation}\label{kernel-blow-up}
\varphi_n(x,t=0)=n \rho (n x),
\end{equation}
where $0 \leq \rho \in C_c(\R)$ is nontrivial. Then 
$
\lim_{n \to \infty} \varphi_n(x,t)=\infty$, $\forall x \in \R, \forall t>0$.
\end{proposition}

See Section \ref{ref-title-precise-instability} for the proof. 

\begin{remark}
A  counterexample  to the $L^1$ continuous dependence for \eqref{example-diffusion} is then given by the sequence of solutions 
$$
\psi_n(x,t):=\varphi_n(x,t)/\sqrt{\|\varphi_n(\cdot,t_0)\|_{L^1}} \quad \mbox{for a fixed $t_0>0$,}
$$
since $\|\psi_n(t=0)\|_{L^1} \to 0$ while $\|\psi_n(\cdot,t)\|_{L^1} \geq \|\psi_n(\cdot,t_0)\|_{L^1} \to \infty$ for any $t \geq t_0$.
\end{remark}

\medskip

In view of the previous results, we now look for a Banach space $X  \subset L^1$ that is strong enough to get  well-posedness for  \eqref{eq:dual} in general. 
We are mainly interested in properly defining an associated semigroup;  see e.g. \cite{BeWi94,BeCrPa01} for a general presentation of nonlinear semigroups.

\begin{definition}
Let $E$ be a normed space. 
\begin{enumerate}[{\rm (a)}]
\item  A family of maps $G_t:E \to E$ parametrized by $t \geq 0$ is a semigroup on $E$ if
$$
\begin{cases}
\mbox{$G_{t=0}=\textup{id}$ (the identity), and} \\
\mbox{$G_{t+s}=G_t G_{s}$ (meaning the composition) for any $t,s \geq 0$.}\\
\end{cases}
$$
\item It is a semigroup of continuous operators if in addition $G_t:E \to E$ is continuous for each $t \geq 0$.
\item And it is strongly continuous  if 
for  each $\varphi_0 \in E$, $t \geq 0 \mapsto G_t \varphi_0 \in E$ is strongly continuous  (i.e. continuous in norm).
\end{enumerate}
\end{definition}

Let  $\varphi$ be the unique viscosity solution of \eqref{eq:dual} and define
\begin{equation}\label{semigroup-HJB}
  G_t: \varphi_0 \in C_b(\R^d) \mapsto \varphi(\cdot,t) \in C_b(\R^d).
\end{equation}
Then $G_t$ is a  semigroup of  Lipschitz continuous (in $C_b$) operators
 by Remark \ref{rem:max}.
 A natural construction is to define $X$ as the completion of some $E \subseteq C_b \cap L^1$, such that $X \subseteq L^1$ and $G_t$ can be extended from $E$ onto $X$. More precisely we  require that
\begin{equation}\label{coarsest-space}
\begin{cases}
\mbox{$E$ is a vector subspace of $C_b \cap L^1(\R^d)$,}\\
\mbox{$E$ is a normed space},\\
\mbox{$E$ is continuously embedded into $L^1(\R^d)$},
\end{cases}
\end{equation}
and for  any data  $(\mathcal{E},b,a)$ satisfying \eqref{bassumption},  
the semigroup \eqref{semigroup-HJB}  satisfies:
\begin{equation}\label{coarsest-semigroup}
\forall t \geq 0, \quad \begin{cases}
\mbox{$G_t(E) \subseteq X:=\overline{E}^{\|\cdot\|_E}$, $G_t:E \to X$ is continuous, and}\\
\mbox{$G_t$ admits an extension onto $X$ as a continuous operator.}
\end{cases}
\end{equation}  
Here $\overline{E}^{\|\cdot\|_E} \subseteq L^1(\R^d)$ is the completion, see Section \ref{sec:notation}.

The best $E$ is given below.

\begin{theorem}[A natural $L^1$ setting for  \eqref{eq:dual}]\label{thm:qualitative-HJB}
The space $C_{b} \cap L^\infty_{\INT}(\R^d)$ is a Banach space satisfying the properties \eqref{coarsest-space}--\eqref{coarsest-semigroup}. Moreover, any other space $E$ satisfying \eqref{coarsest-space}--\eqref{coarsest-semigroup} is continuously embedded into  $C_{b} \cap L^\infty_{\INT}(\R^d)$.
\end{theorem}

\begin{remark}
Since  the best $E=X$ is  a Banach space by Theorem \ref{thm:uniformly-integrable}, it is a posteriori not necessary to extend $G_t$ outside $C_b$.  The classical notion of viscosity solutions is then already satisfactory to study $L^1$ solutions of fully nonlinear degenerate PDEs.  
\end{remark}

Theorem \ref{thm:qualitative-HJB} relies on the following estimate:

\begin{theorem}[General $L^\infty_\INT$ stability]
\label{thm:quantitative-HJB-unif}
Assume \eqref{bassumption} and $T \geq 0$. For any bounded subsolution $\varphi$ and supersolution $\psi$ of \eqref{eq:dual-1}, 
\begin{equation}\label{esti:quantitative-HJB-unif}
\int \sup_{\overline{Q}_{1}(x) \times [0,T]} \left(\varphi^\ast-\psi_\ast\right)^+ \dd x 
\leq C \int \sup_{\overline{Q}_{1}(x)} \left(\varphi^\ast-\psi_\ast\right)^+(\cdot,0) \dd x,
\end{equation}
for some constant $C=C(d,\|a\|_\infty,\|b\|_\infty,T) \geq 0$.
\end{theorem} 

As a consequence we have the following result:

\begin{corollary}[$L^\infty_\INT$ well-posedness of \eqref{eq:dual}]\label{cor-nv}
Assume \eqref{bassumption} and $G_t$ is the solution semigroup defined in \eqref{semigroup-HJB}. Then its restriction to $C_b \cap L^\infty_\INT(\R^d)$ is a strongly continuous semigroup  of Lipschitz continuous operators.   
\end{corollary}

The proofs of Theorem \ref{thm:quantitative-HJB-unif} and Corollary \ref{cor-nv} are given in Section \ref{esp:ref-proof-bis-bis},  while Theorem \ref{thm:qualitative-HJB} is proved  in Section \ref{ref-title-precise-instability}.

\subsubsection*{A certain duality between nonlinear semigroups}

For each $t \geq 0$,  let 
$$
S_t:u_0 \in L^\infty(\R^d) \mapsto u(\cdot,t) \in L^\infty(\R^d)
$$ 
where $u$ is the entropy solution of \eqref{E}, and  let 
$$
G_t: \varphi_0 \in C_{b} \cap L^\infty_{\INT}(\R^d) \mapsto \varphi(\cdot,t) \in C_{b} \cap L^\infty_{\INT}(\R^d)
$$
where $\varphi$ is the viscosity solution of \eqref{eq:dual-true}.  Note that $G_t=G^{m,M}_t$ depends on the parameters $m$ and $M$ through  Equation \eqref{eq:dual-true-1}.  

\begin{corollary}[A form of duality]\label{cor:sg}
Assume \eqref{fassumption}, $m<M$, and  consider the semigroups $S_t$ and $G_t$  defined as above. Then $G_t$ is the smallest strongly continuous semigroup  of continuous  operators on $C_{b} \cap L^\infty_{\INT}(\R^d)$ satisfying
\begin{equation}\label{etoile-bis}
\int_{\R^d} |S_tu_0-S_tv_0| \varphi_0 \dd x \leq \int_{\R^d} |u_0-v_0| G_t \varphi_0 \dd x,
\end{equation}
for every $u_0$ and $v_0$ in $L^\infty(\R^d,[m,M])$, $0 \leq \varphi_0\in C_{b} \cap L^\infty_{\INT}(\R^d)$, and $t \geq 0$. 
\end{corollary}

The proof of Corollary \ref{cor:sg} is given in Section \ref{sec:proofs-bis}. 

\begin{remark}
Here ``smallest'' means that any other semigroup $H_t$ satisfying the same properties is such that
$$
G_t\varphi_0\leq H_t\varphi_0 \quad \forall \varphi_0 \geq 0, 
{\forall t \geq 0.}
$$
\end{remark}

\begin{remark}\label{noq}
\begin{enumerate}[\rm (a)]
\item\label{item-referee} Inequality \eqref{etoile-bis} can be seen as a nonlinear dual inequality between $S_t$ and $G_t$, and $G_t$ as a dual semigroup of $S_t$ whose restriction over the cone $C_b \cap L^\infty_\INT(\R^d,\R^+)$ is  entirely determined by $S_t$ through \eqref{etoile-bis}.
\item\label{item-oq-referee} The question of duality in the other direction is
  open. 
Let us formulate it precisely. Consider $S_t$  and the whole family  $\{G_t^{m,M}:m<M\}$ defined just before
Corollary \ref{cor:sg}.
\begin{quote}
{\bf Open question.} \it Is $S_t$ the {\bf unique} weakly-$\star$ continuous semigroup on $L^\infty(\R^d)$ such that for all $m<M$, $G_t^{m,M}$ is the smallest strongly continuous semigroup of continuous operators on $C_{b} \cap L^\infty_{\INT}(\R^d)$ satisfying
\begin{equation}\label{etoile-bis-bis}
\int_{\R^d} |S_tu_0-S_tv_0| \varphi_0 \dd x \leq \int_{\R^d} |u_0-v_0| G^{\scriptscriptstyle m,M}_t \varphi_0 \dd x,
\end{equation}
for all $u_0$ and $v_0$ in $L^\infty(\R^d,[m,M])$, $0
\leq\varphi_0 \in C_{b} \cap L^\infty_{\INT}(\R^d)$, and $t \geq
0$?
\end{quote}
A positive answer would mean that $S_t$ is conversely entirely determined by the family $\{G_t^{m,M}:m<M\}$ 
through \eqref{etoile-bis-bis}.

\item Following part \eqref{item-referee}, we might be tempted to define a notion of dual for more general nonlinear semigroups. It is not our aim to explore such a direction, but note however that it would make sense only if
\begin{itemize}
\item[(i)] we have a reciprocal duality as discussed in part \eqref{item-oq-referee}, and
\item[(ii)] we can recover standard duality notions in the linear case.
\end{itemize} 
We can say more about (ii), and in 
\ref{app-referee} we give a sample result for slightly more abstract semigroups, for which we would not a priori know the associated equations.
\end{enumerate}
\end{remark}

\section{Proofs}\label{sec:proofs}
This section is devoted to the proofs of the results of Section \ref{main-results}. We will prove them in a certain order to arrive at Corollaries \ref{cor:opt} and \ref{cor:sg}, thus concluding by the optimality of the weight and the interpretation in terms of dual nonlinear semigroup. The proofs of Propositions 
\ref{natural-second} and \ref{lem:blow-up} and Theorem \ref{thm:qualitative-HJB} are  independent of this developement and given  at the end of the section. 

\subsection{More on viscosity solutions of \eqref{eq:dual}}

We need further classical results that can be found in \cite{CrIsLi92,FlSo93,Bar94,BaCa-Do97}. 

\begin{proposition}[Stability  w.r.t.  sup]\label{prop:ssi}
Assume \eqref{bassumption} and $\mathcal{F} \neq \emptyset$ is a uniformly locally  bounded family of viscosity subsolutions of \eqref{eq:dual-1}. Then, the function 
$$
(x,t) \mapsto \sup \{ \varphi(x,t):\varphi \in \mathcal{F} \}
$$ 
is a viscosity subsolution of \eqref{eq:dual-1}.
\end{proposition} 

The next results concern relaxed limits; cf. \eqref{relaxed-referee}.

\begin{proposition}[Stability  w.r.t.  relaxed limits]\label{prop:rsl}
Assume \eqref{bassumption} and  let  $(\varphi_\veps)_{\veps>0}$  be  a family of uniformly  locally bounded viscosity subsolutions (resp. supersolutions) of \eqref{eq:dual-1}. Then $\wlimsup \varphi_\veps$ (resp. $\wliminf \varphi_\veps$)
is a subsolution of \eqref{eq:dual-1} (resp. supersolution).
\end{proposition}

\begin{remark}
The notion of solution  (or semisolution)  is thus stable under local uniform convergence (equivalent to $\wlimsup \varphi_\veps=\wliminf \varphi_\veps$).
\end{remark}

\begin{proposition}[Limiting initial data]\label{prop:rsl-bis}
Assume \eqref{bassumption} and $(\varphi_\veps)_{\veps>0}$ is a 
uniformly locally bounded 
family of viscosity subsolutions (resp. supersolutions) of \eqref{eq:dual-1}. Then $\wlimsup \varphi_\veps$ (resp. $\wliminf \varphi_\veps$) satisfies
\begin{equation*}
\wlimsup \varphi_\veps (x,0) = \wlimsup \big[(\varphi_\veps)^\ast(\cdot,0) \big](x) \quad \forall x \in \R^d 
\end{equation*} 
(resp. $\wliminf \varphi_\veps (x,0) = \wliminf \big[(\varphi_\veps)_\ast(\cdot,0) \big](x)$).
\end{proposition}

\begin{remark}\label{limt0}
For subsolutions this means that
$$
\limsup_{\substack{\R^d \times \R^+ \ni (y,s)\to(x,0) \\ \veps \to 0^+}} \varphi_\veps(y,s) = \limsup_{\substack{\R^d \ni y\to x \\ \veps \to 0^+}} (\varphi_\veps)^\ast(y,0),
$$
where $(\varphi_\varepsilon)^\ast$ is the upper semicontinuous envelope computed in $(x,t)$.
The proof can be found in \cite{BaPe88} and 
\cite[Theorem 4.7]{Bar94}. The idea is to first consider 
$\varphi
:=
\wlimsup \varphi_\veps$, 
$\varphi_0(x)
:= 
\wlimsup \big[(\varphi_\veps)^\ast(\cdot,0)\big](x)$, and show that $\min \{\partial_t \varphi-H(D \varphi,D^2 \varphi),\varphi-\varphi_0\} \leq 0$  at $t=0$  in the viscosity sense. Fix  then some $x$ and use the viscosity inequalities at a max $(\overline{y},\overline{t})$ of the function $\varphi(y,t)-|y-x|^2/
\tilde{\varepsilon}
-C t$ with $C$ large enough such that $\overline{t}=0$. We get $\varphi(x,0) \leq \varphi_0(\overline{y})$ and conclude as $\tilde{\varepsilon} \to 0^+$.
\end{remark}

Here is  the stability for minimal  solutions, see 
\ref{app:min-max} for the proof.

\begin{proposition}[Stability of  minimal  solutions]
\label{stability-extremal}
Assume \eqref{bassumption} and $(\varphi_0^n)_n$ is a non\-decreasing uniformly globally bounded sequence. If $\underline{\varphi}_n$ is the  minimal solution of \eqref{eq:dual} with $\varphi_0^n$ as initial data, then 
$
\sup_n \underline{\varphi}_n 
$
is the minimal solution of \eqref{eq:dual} with initial data  $
\sup_n (\varphi_0^n)_\ast.
$
\end{proposition}

Let us continue with regularization procedures.  Usually we consider inf and supconvolutions, but for convex  Hamiltonians  we can use the classical convolution for supersolutions, see \cite{BaJa02,BaJa07} (the ideas were introduced in \cite{Kry97}).

\begin{lemma}
\label{lem:convolution}
Assume \eqref{bassumption}, $\varphi \in BLSC(\R^d \times (0,\infty))$ is a supersolution of \eqref{eq:dual-1}, and $0 \leq f \in L^1(\R^d \times (-\infty,0))$. Then $\varphi \ast_{x,t} f$ is a supersolution of \eqref{eq:dual-1}.
\end{lemma}

Below is another version that will be needed. 

\begin{lemma}
\label{lem:convolution-bis} 
Assume \eqref{bassumption}, $\varphi \in C_b(\R^d \times (0,\infty))$ is a
supersolution of \eqref{eq:dual-1}, and $0 \leq g \in
L^1(\R^d)$. Then  $\varphi \ast_x g$ remains
a supersolution. 
\end{lemma}

The latter lemma is not proven in \cite{BaJa02,BaJa07}, but can be obtained via a standard approximation procedure. Let us give it for completeness. 

\begin{proof}[Proof of Lemma \ref{lem:convolution-bis}]
By Lemma \ref{lem:convolution}, $\varphi_{\nu}:=\varphi \ast_x g
\ast_x \rho_\nu \ast_t \theta_\nu$ is a supersolution of \eqref{eq:dual-1}. It remains to pass to the limit as $\nu \to
0^+$. We will show that the convergence  is local uniform towards
$\varphi \ast_x g$, which will be sufficient by stability of the equation. We only need a local uniform convergence for $t>0$ because the conclusion concerns the PDE only. With the assumed regularity on $\varphi$, 
$$
\lim_{\nu \to 0^+} \varphi \ast_{x} \rho_\nu \ast_t \theta_\nu= \varphi \quad \mbox{locally uniformly,}
$$
and $\|\varphi \ast_{x} \rho_\nu \ast_t \theta_\nu\|_\infty \leq \|\varphi\|_\infty$. Moreover, for any $x \in \R^d$, $t > 0$ and $R \geq 0$,
\begin{equation*}
\begin{split}
|\varphi_\nu-\varphi \ast_x g|(x,t) & \leq |\varphi \ast_{x} \rho_\nu \ast_t \theta_\nu-\varphi| \ast_x g(x,t)\\
& \leq \left(\sup_{|y| \leq R} |\varphi \ast_{x} \rho_\nu \ast_t \theta_\nu-\varphi|(x-y,t)\right) \int_{|y| \leq R} g(y) \dd y\\
& \quad + 2 \|\varphi\|_\infty \int_{|y| > R} g(y) \dd y.
\end{split}
\end{equation*}
This is enough to conclude since $\lim_{R \to \infty} \int_{|y| > R} g(y) \dd y=0$.
\end{proof}

\subsection{$L^\infty_\INT$ well-posedness: Proofs of Theorem \ref{thm:quantitative-HJB-unif} and Corollary \ref{cor-nv}}\label{esp:ref-proof-bis-bis}

Let us now show that \eqref{eq:dual} is well-posed in $L^\infty_\INT$ as stated in Corollary \ref{cor-nv}. We first need to prove Theorem \ref{thm:quantitative-HJB-unif} for which we will use the lemmas below.  

\begin{lemma}\label{lem:subsolutionsum}
Assume \eqref{bassumption}, and $\varphi$ and $\psi$ are sub and supersolutions of \eqref{eq:dual-1}. Then $(\varphi^\ast-\psi_\ast)^+$ remains a subsolution.
\end{lemma}

\begin{proof}[Sketch of proof]
First note that $\varphi-\psi$ is a subsolution since 
\begin{equation*}
\partial_t(\varphi-\psi)  \leq \sup_{\xi \in \mathcal{E}} H_\xi(\varphi)-\sup_{\xi \in \mathcal{E}}H_\xi(\psi) \leq \sup_{\xi \in \mathcal{E}} \left(H_\xi(\varphi)-H_\xi(\psi) \right),
\end{equation*}
for $
H_\xi(\varphi):=b(\xi) \cdot D\varphi+\tr \left(a(\xi) D^2 \varphi \right)$. Since $(\varphi^\ast-\psi_\ast)^+=\max \{\varphi^\ast-\psi_\ast,0\}$, it is a subsolution by
 stability of viscosity subsolutions w.r.t. $\max$, 
 see  Proposition~\ref{prop:ssi}. 
\end{proof}

The rigorous justification of the above computations can be done by using a test function,  Ishii lemma, and semijets \cite[Theorem 8.3]{CrIsLi92}. The details are standard and left to the reader.  Here is a second lemma involving the  profile 
\begin{equation*}
U: r \geq 0 \mapsto c_0 \int_r^\infty \e^{-\frac{s^2}{4}} \dd s,
\end{equation*}
where $c_0>0$ is  chosen  such that $U(0)=1$. 

\begin{lemma}
\label{esti-with-decay}
 Let $L_b \geq 0$ and $L_a>0$.   For any $(x,t) \in \R^d \times \R^+$, define
\begin{equation}\label{def-phi-sp}
\Psi(x,t):=
\begin{cases}
U \left(\left(|x|-1-L_b t\right)^+/\sqrt{L_a t}\right) & \mbox{if $t >0$},\\
\mathbf{1}_{|x|<1} & \mbox{if $t=0$}.
\end{cases}
\end{equation}
Then in the viscosity sense,
$$
\partial_t \Psi \geq L_b |D\Psi|+L_a \sup_{\lambda \in \textup{Sp} (D^2 \Psi)} \lambda^+ \quad \mbox{in } \R^d \times (0,\infty)
$$
 (it is in fact an equality). Moreover
$
\Psi \in C_b(\R^d \times (0,\infty)) \cap C(\R^+;L^1(\R^d))
$ where the latter time continuity holds up to $t=0$. 
\end{lemma}

\begin{remark}
 Roughly speaking, we will  use $\Psi$ as a fundamental solution to construct $L^1$ supersolutions of \eqref{eq:dual}, but we cannot take it as a Dirac mass at $t=0$ because of Proposition \ref{lem:blow-up}.
\end{remark}

\begin{proof}  The  desired PDE holds if $|x| < 1+L_b t$ since $\Psi$ is constant in that region. It is also satisfied if $|x|=1+L_b t$ because the subjets are empty. Now if $|x| >1+L_b t$, then
\begin{equation*}
\partial_t \Psi = -L_a \frac{|x|-1-L_b t}{2 (L_a t)^{\frac{3}{2}}} U'-\frac{L_b}{\sqrt{L_at}} U', \quad D \Psi =\frac{x}{|x|} \frac{U'}{\sqrt{L_at}},
\end{equation*}
and
\begin{equation*}
\partial_{x_i x_j}^2 \Psi=\left(\frac{\delta_{ij}}{|x|}-\frac{x_i x_j}{|x|^3} \right) \frac{U'}{\sqrt{L_a t}} +\frac{x_i x_j}{|x|^2} \frac{U''}{L_a t}.
\end{equation*}
Since $U' \leq 0$ and $U'' \geq 0$, we have 
$
\sum_{i,j=1}^d \partial_{x_i x_j}^2 \Psi h_ih_j \leq \frac{U''}{L_a t}
$
for any $h=(h_i)$ with $|h|=1$. Hence $\sup_{\lambda \in \textup{Sp} (D^2 \Psi)} \lambda^+ \leq \frac{U''}{L_a t}$
and 
$$
\partial_t \Psi-  L_b |D\Psi|-L_a \sup_{\lambda \in \textup{Sp} (D^2 \Psi)} \lambda^+\geq-\frac{rU' (r)/2 +U''(r)}{t}
$$
with $r=(|x|-1-L_b t)/ \sqrt{L_a t}$.  The  right-hand side is zero by definition of $U$,  and we obtain the desired equation for positive times. Now the detailed verification that $
\Psi \in C_b(\R^d \times (0,\infty)) \cap C(\R^+;L^1(\R^d))
$ does not contain any particular difficulty and is left to the reader. The proof is complete.
\end{proof}

\begin{proof}[Proof of Theorem \ref{thm:quantitative-HJB-unif}]
Let $L_b:=\|b\|_\infty$ and $L_a:=\|\tr (a)\|_\infty$ and assume $L_a>0$.
 We will use the following Ky Fan inequality \cite{The75}: 
\begin{equation}\label{kyfan}
\tr \left(X Y \right) \leq \sum_{i=1}^d \lambda_i(X) \lambda_i(Y) \quad \forall X, Y \text{ real $d \times d$ symmetric matrices},
\end{equation}
with the ordered eigenvalues $\lambda_1 \leq \dots \leq \lambda_d$. It implies that
any subsolution of \eqref{eq:dual-1} is a subsolution of the equation
\begin{equation}\label{simple-proof}
\partial_t \varphi = L_b |D\varphi|+L_a \sup_{\lambda \in \textup{Sp} (D^2 \varphi)} \lambda^+.
\end{equation}
 Consider now  arbitrary  bounded sub and supersolutions $\varphi$ and $\psi$ of \eqref{eq:dual-1}. By Lemma \ref{lem:subsolutionsum}, $(\varphi^\ast-\psi_\ast)^+$ is a subsolution of \eqref{eq:dual-1} thus of \eqref{simple-proof}.   To prove  Estimate  \eqref{esti:quantitative-HJB-unif}, we will construct an integrable supersolution of \eqref{simple-proof}.
 We will take it of the form 
\begin{equation*}
\psi:=\Psi \ast_x \sup_{\overline{Q}_1(\cdot)} \phi_0,
\end{equation*}
where 
$
\phi_0(x):=\left(\varphi^\ast-\psi_\ast\right)^+(x,t=0)
$
and $\Psi$ is defined in Lemma \ref{esti-with-decay}.
Let us use Lemma  \ref{lem:convolution-bis}  to show that $\psi$ is a supersolution of \eqref{simple-proof}. We need 
$
\sup_{\overline{Q}_1(\cdot)} \phi_0$ to be integrable, and this can be assumed without loss of generality since \eqref{esti:quantitative-HJB-unif} trivially holds if not. Now recalling that $\Psi\in C_b(\R^d \times (0,\infty))$ is a supersolution of \eqref{simple-proof}, Lemma \ref{lem:convolution-bis} applies and $\psi$ remains a supersolution. Since moreover 
$
\Psi \in C(\R^+;L^1(\R^d))
$
and $\sup_{\overline{Q}_1(\cdot)} \phi_0 \in L^\infty(\R^d)$, this supersolution is continuous up to $t=0$ and satisfies 
$$
\psi(x,0) =\int \mathbf{1}_{|y| <1} \sup_{\overline{Q}_1(x-y)} \underbrace{\phi_0}_{=(\varphi^\ast-\psi_\ast)^+(t=0)} \dd y \geq \left(\varphi^\ast-\psi_\ast\right)^+(x,0).
$$
Hence $(\varphi^\ast-\psi_\ast)^+ \leq \psi$ everywhere by the comparison principle, and 
\begin{equation*}
\begin{split}
\int \sup_{\overline{Q}_1(x) \times [0,T]} \left(\varphi^\ast-\psi_\ast\right)^+ \dd x & \leq \int \sup_{\overline{Q}_1(x) \times [0,T]} \psi \dd x\\
& \leq \int \sup_{t \in [0,T]} \Psi(y,t) \dd y \int \sup_{\overline{Q}_2(x)}\phi_0 \dd x,
\end{split}
\end{equation*}
by 
the Fubini  theorem,  etc. The first integral satisfies
\begin{equation*}
\int \sup_{t \in [0,T]} \Psi(y,t) \dd y\leq\int U \left(\left(|y|-1-L_bT\right)^+/\sqrt{L_a T}\right) \dd y<\infty,
\end{equation*}
by 
\eqref{def-phi-sp} 
and since $U$ is nondecreasing and integrable. For the second integral, Lemma \ref{lem:criterion} implies that 
\begin{equation*}
\int \sup_{\overline{Q}_2(x)}\phi_0 \dd x \leq C  \int \sup_{\overline{Q}_{1}(x)} \left(\varphi^\ast-\psi_\ast\right)^+(\cdot,0) \dd x,
\end{equation*}
for a constant $C$ which only depends on $d$. Combining the three inequalities above completes the proof of \eqref{esti:quantitative-HJB-unif} when $L_a=\|\tr (a)\|_\infty>0$. If $L_a=0$, there is no diffusive part in \eqref{eq:dual-1}, and \eqref{esti:quantitative-HJB-unif}
follows from Proposition \ref{representation-deterministic} and Lemma \ref{lem:criterion}.
\end{proof}

We are ready to prove Corollary \ref{cor-nv}. We need the result below.

\begin{lemma}\label{lem:absolute-value}
Assume \eqref{bassumption} and $\varphi$ and $\psi$ are continuous viscosity solutions of \eqref{eq:dual-1}. Then $|\varphi-\psi|$ is a subsolution of the same PDE.
\end{lemma}

\begin{proof}
 Use that 
$
|\varphi-\psi|=\max \{(\varphi-\psi)^+,(\psi-\varphi)^+\}
$ and Lemma \ref{lem:subsolutionsum}. 
\end{proof}

\begin{proof}[Proof of Corollary \ref{cor-nv}]
The fact that $G_t$ maps $C_{b} \cap L^\infty_{\INT}(\R^d)$ into itself follows from Theorem \ref{thm:quantitative-HJB-unif}. Indeed, if $\varphi_0 \in C_{b} \cap L^\infty_{\INT}(\R^d)$, then the function $(x,t) \mapsto |G_t \varphi_0(x)|$ is a bounded subsolution of \eqref{eq:dual-1}, by Lemma \ref{lem:absolute-value} with $\psi \equiv 0$. Estimate  \eqref{esti:quantitative-HJB-unif} then implies that for any $t \geq 0$, 
\begin{equation*}
\|G_t \varphi_0\|_\INT = \int \sup_{\overline{Q}_1(x)} |G_t \varphi_0| \dd x 
\leq 
C 
\int \sup_{\overline{Q}_1(x)} |\varphi_0| \dd x,
\end{equation*}
for some constant $C=C(d,\|a\|_\infty,\|b\|_\infty,t)$. Let us now prove that 
$$
G_t:C_{b} \cap L^\infty_{\INT}(\R^d) \to C_{b} \cap L^\infty_{\INT}(\R^d)
$$ 
is  Lipschitz  continuous for any $t \geq 0$. Let us apply again \eqref{esti:quantitative-HJB-unif} to 
$$
(x,t) \mapsto |G_t \varphi_0(x)-G_t \psi_0(x)|,
$$ 
which is a subsolution of \eqref{eq:dual-1} by Lemma \ref{lem:absolute-value}. As above we get that 
\begin{equation*}
\left\|G_t \varphi_0-G_t\psi_0 \right\|_\INT \leq C 
\int  \sup_{\overline{Q}_1(x)} |\varphi_0-\psi_0| \dd x,
\end{equation*}
and deduce the desired continuity because $C$ does not depend 
on the initial data. 
 Hence  $G_t$ is a semigroup of  Lipschitz  continuous operators on $C_b \cap L^\infty_\INT(\R^d)$  and it  remains to prove the time continuity. Fix $t_0 \geq 0$ and let us show that 
$$
\int \sup_{\overline{Q}_1(x)} |G_t\varphi_0-G_{t_0} \varphi_0| \dd x \to 0 \quad \mbox{as } t \to t_0.
$$
The pointwise convergence follows from the continuity of $(x,t) \mapsto G_t \varphi_0(x)$ (as continuous solution of \eqref{eq:dual}), and a dominating function is given by
$$
x \mapsto \sup_{(y,s) \in \overline{Q}_1(x) \times [0,t_0+1]} 2|G_s \varphi_0(y)|
$$
which is integrable by Theorem \ref{thm:quantitative-HJB-unif}. 
\end{proof}

\subsection{Weighted $L^1$ contraction: Proof of Theorem \ref{thm:weighted}}
\label{sec:contraction}

We  continue with  the general weighted $L^1$ contraction principle for \eqref{E}.

\begin{proof}[Proof of Theorem \ref{thm:weighted}]
We have to show that  
\begin{equation}
\label{nonlinear-dual-bis}
\int_{\R^d} |u-v|(x,T) \varphi_0(x) \dd x \leq \int_{\R^d} |u_0-v_0|(x) \underline{\varphi}(x,T)\dd x \quad \forall T \geq 0.
\end{equation}
Let us use the Kato inequality \eqref{dualeq} with $0\leq\phi\in C_c^\infty(\R^d\times[0,T])$. We then obtain, for  a.e.  $x \in \R^d$ and $t \geq 0$, 
\begin{equation}\label{kato-tech-preaccurate}
\begin{split}
& \left\{\sum_{i=1}^d q_i(u,v)\partial_{x_i} \phi+\sum_{i,j=1}^d r_{ij}(u,v)\partial_{x_ix_j}^2\phi\right\}(x,t) \\
&  =  
\sgn(u(x,t)-v(x,t)) \int_{v(x,t)}^{u(x,t)} \left\{F'(\xi) \cdot D \phi(x,t)+\tr \left( A(\xi) D^2 \phi (x,t)\right)  \right\} \dif \xi\\
& \leq |u(x,t)-v(x,t)| \esssup_{m \leq \xi \leq M} \left\{F'(\xi) \cdot D \phi(x,t)+\tr \left( A(\xi) D^2 \phi (x,t)\right)  \right\},
\end{split}
\end{equation}
where we have taken the sup over $[m,M]$ because of the maximum principle Lemma \ref{lem:max-comp}. Injecting into \eqref{dualeq}, we get that
\begin{multline}\label{kato-tech-accurate}
\int_{\R^d} |u-v|(x,T)\phi(x,T)\dd x \leq \int_{\R^d}|u_0-v_0|(x)\phi(x,0)\dd x\\
+\iint_{\R^d \times (0,T)} |u-v| \left(\partial_t \phi+\esssup_{m \leq \xi \leq M} \left\{F'(\xi) \cdot D \phi+\tr \left( A(\xi) D^2 \phi\right)  \right\}\right) \dif x \dd t.
\end{multline} 
In the third integral, we recognize the backward 
in time
version of \eqref{eq:dual-true-1}. The proof of \eqref{nonlinear-dual-bis} then consists in taking $\phi(x,t) 
=
\underline{\varphi}(x,T-t)$.

\medskip

{\bf 
Simplified 
case:} $0 \leq \varphi_0 \in C_c (\R^d)$.\\ 
Now \eqref{eq:dual-true} has a unique viscosity solution $\varphi$ which coincides with $\underline{\varphi}$.  It belongs to $C_b(\R^d \times \R^+) \cap C(\R^+;L^1(\R^d))$ by Corollary \ref{cor-nv} and Theorem \ref{thm:uniformly-integrable}. 
Let us regularize it by convolution
$$
\varphi_\nu:= \varphi \ast_{x,t} \left( \rho_\nu \theta_\nu\right),
$$
with the mollifiers \eqref{space-approx} and \eqref{time-approx}. It follows that
$$
\varphi_\nu \in C^\infty(\R^d \times \R^+) \cap C(\R^+;L^1(\R^d))
$$ 
along with all its derivatives. This is enough to take $\phi_\nu(x,t):=\varphi_\nu(x,T-t)$ as a test function in \eqref{kato-tech-accurate} by approximation. Note that  $\phi_\nu$  is a supersolution of the backward version of \eqref{eq:dual-true-1} by Lemma \ref{lem:convolution}, i.e.
$$
\partial_t \phi_\nu+\esssup_{m \leq \xi \leq M} \left\{F'(\xi) \cdot D \phi_\nu+\tr \left( A(\xi) D^2 \phi_\nu\right)\right\}  \leq 0 \quad \mbox{for any $t<T$.}
$$
Inequality \eqref{kato-tech-accurate} with the test function $\phi_\nu$ then implies that
$$
\int_{\R^d} |u-v|(x,T) \varphi_\nu(x,0)\dd x \leq \int_{\R^d}|u_0-v_0|(x) \varphi_\nu(x,T)\dd x,
$$
for any $T \geq 0$ and $\nu >0$. By the $C(\R^+;L^1(\R^d))$ regularity of $\varphi$, the convolution $\varphi_\nu=\varphi \ast_{x,t} (\rho_\nu \theta_\nu)$ converges to $\varphi$ in $C([0,T];L^1(\R^d))$ as $\nu \to 0^+$. Passing to the limit as $\nu \to 0^+$ then yields \eqref{nonlinear-dual-bis}.

\medskip

{\bf General case:} $0 \leq \varphi_0 \in BLSC(\R^d)$.\\
We would like to pointwise approximate $\varphi_0$ by a monotone sequence $\varphi_0^n \uparrow \varphi_0$ such that $0 \leq \varphi_0^n \in C_c(\R^d)$. Take 
$$
\varphi_0^n(x):= \inf_{y \in \R^d} \left\{ \varphi_0(y) \mathbf{1}_{|y| <n}+n|x-y|^2  \right\} \geq 0.
$$
Then $\varphi_0^n$ is continuous as an infconvolution,  see e.g. \cite{CrIsLi92,FlSo93,Bar94,BaCa-Do97}.  Also, 
$$
\varphi_0^n(x) \leq \varphi_0(x) \mathbf{1}_{|x| <n} \quad \forall x \in \R^d,
$$ 
which implies that $\varphi_0^n \in C_c(\R^d)$. In the limit $n \to \infty$, we have $\varphi_0^n \uparrow (\varphi_0)_\ast=\varphi_0$. Let $\varphi_n$ be the solution of \eqref{eq:dual-true} with initial data $\varphi_0^n$, then by the previous step, 
$$
\int_{\R^d} |u-v|(x,T) \varphi_0^n(x)\dd x \leq \int_{\R^d}|u_0-v_0|(x) \varphi_n(x,T)\dd x,
$$
for any $T \geq 0$ and $n$. By the stability of  minimal  solutions (see Proposition \ref{stability-extremal}), these solutions satisfy $\varphi_n \uparrow \underline{\varphi}$ pointwise.  So we conclude the proof of \eqref{nonlinear-dual-bis} by passing to the limit as $n \to \infty$ using the monotone convergence theorem.
\end{proof}

\begin{remark}
Going back to \eqref{dualeq} and \eqref{kato-tech-preaccurate}, we might think about the kinetic setting for \eqref{E} since
\begin{align*}
&\sgn(u(x,t)-v(x,t)) \int_{v(x,t)}^{u(x,t)} \left\{\partial_t\phi(x,t)+F'(\xi) \cdot D \phi(x,t)+\tr \left( A(\xi) D^2 \phi (x,t)\right)  \right\} \dif \xi\\
&= \int_{\R}|\chi(\xi;u)-\chi(\xi;v)| \left\{\partial_t\phi(x,t)+F'(\xi) \cdot D \phi(x,t)+\tr \left( A(\xi) D^2 \phi (x,t)\right)  \right\} \dif \xi,
\end{align*}
with the usual kinetic function $\chi$; cf. \cite{ChPe03}. However, we did not explore this.  For $L^1$ kinetic solutions of \eqref{E}, $u$ and $v$ would take values outside any bounded interval, so there would be further terms for large $|\xi|$ and we do not have any idea of what might then be a reasonable version of \eqref{weighted-intro}.
\end{remark}

\subsection{Duality: Proofs of Theorem \ref{thm:dual} and Corollaries \ref{cor:opt} and \ref{cor:sg}}
\label{sec:proofs-bis}

Let  us now establish the new characterization of viscosity supersolutions (Theorem \ref{thm:dual}). We need several technical lemmas.  

Here is a first classical result on entropy solutions.

\begin{lemma}\label{lem:distribution} Assume \eqref{fassumption} and $u_0 \in L^\infty(\R^d)$. Then, the entropy solution of \eqref{E} is a distributional solution 
of \eqref{E},
\begin{align*}
& \iint_{\R^d \times \R^+} \left(u \partial_t\phi
 +\sum_{i=1}^d F_i(u)\partial_{x_i}\phi+\sum_{i,j=1}^{d} \mathcal{A}_{ij}(u)\partial_{x_ix_j}^2\phi\right)\dif x\dd t \\
& +\int_{\R^d} u_0(x) \phi(x,0) \dd x=0 \quad \forall \phi \in C^\infty_c(\R^d \times \R^+),
\end{align*}
where $\mathcal{A}_{ij}(u)
 = 
\int_0^u A_{ij}(\xi) \dd \xi$.
\end{lemma}

\begin{proof}
Take $\eta(u)=\pm u$ successively in the entropy inequalities, Definition \ref{defentropy}\eqref{item:entropy}.
\end{proof}

Here is another result on the continuity in time.

\begin{lemma}\label{lem:L1}
Assume \eqref{fassumption}, $u_0, v_0 \in L^\infty(\R^d)$ with $u_0-v_0 \in L^1(\R^d)$, $u$ and $v$ entropy solutions of \eqref{E} with initial data $u_0$ and $v_0$. Then $u-v \in C(\R^+;L^1(\R^d))$.
\end{lemma}

\begin{proof}[Proof of Lemma \ref{lem:L1}]
By Theorem \ref{thm:weighted} with $\varphi_0 \equiv 1$, we have
$$
\|u(\cdot,t)-v(\cdot,t)\|_{L^1} \leq \|u_0-v_0\|_{L^1} \quad \forall t \geq 0.
$$
Since the left-hand side is finite, $u-v \in L^\infty(\R^+;L^1(\R^d))$. By the continuity in time with values in $L^1_{\textup{loc}}(\R^d)$ of these functions, it remains to prove that
\begin{equation}
\label{cont-to-prove}
\lim_{R \to \infty} \sup_{t \in [0,T]} \int_{|x| \geq R} |u(x,t)-v(x,t)| \dd x=0 \quad \forall T \geq 0.
\end{equation}
To do so, we will use again Theorem \ref{thm:weighted}. 

Fix $m<M$ such that $u_0$ and $v_0$ take their values in $[m,M]$, and consider 
$$
\varphi_0^R(x):=\varphi_0\left(\frac{x}{R}\right), \quad R>0,
$$ 
where $\varphi_0=\varphi_0(x)$ is some kernel such that
$$
\begin{cases}
0 \leq  \varphi_0 \in C_b(\R^d),\\
\varphi_0(x)=0 \mbox{ for } |x| \leq 1/2,\\
\mbox{and } \varphi_0(x)=1  \mbox{ for } |x| \geq 1.
\end{cases}
$$
With that choice, $\varphi_0^R \to 0$ as $R \to \infty$ locally uniformly in $\R^d$. We then claim that the solutions $\varphi_R$ of \eqref{eq:dual-true} with initial data $\varphi_0^R$ converge locally uniformly in $\R^d\times \R^+$ to zero too. This is a consequence of the method of relaxed semilimits \cite{BaPe88}. Let us give details for completeness. By the maximum principle, 
$$
\|\varphi_R\|_\infty \leq \|\varphi_0^R\|_\infty=\|\varphi_0\|_\infty \quad \forall R>0.
$$
We can then apply Propositions \ref{prop:rsl} and \ref{prop:rsl-bis} to $\wlimsup \varphi_R$ as $R \to \infty$ and get that it is a subsolution of \eqref{eq:dual-true-1} satisfying
$$
\wlimsup \varphi_R(x,0) = \wlimsup \varphi_0^R(x)=0 \quad \forall x \in \R^d.
$$
Let us recall that the above $\wlimsup \varphi_0^R$ as $R \to \infty$ is only taken in space; \normalcolor cf. \eqref{relaxed-referee} and Remark \ref{limt0}. Similarly $\wliminf \varphi_R$ as $R \to \infty$ is a supersolution of \eqref{eq:dual} with zero as initial data. The comparison principle then implies that
$$
\wlimsup \varphi_R \leq \wliminf \varphi_R.
$$ 
Hence $\varphi_R$ converges locally uniformly in $\R^d\times \R^+$, as $R \to \infty$, to the unique solution of \eqref{eq:dual-true} with zero initial data, that is zero itself. 

Now we can show \eqref{cont-to-prove}. By Theorem \ref{thm:weighted} with the previous $m$, $M$, and $\varphi_0^R$, 
\begin{multline*}
\int_{|x| \geq R} |u(x,t)-v(x,t)| \dd x \leq \int_{\R^d} |u(x,t)-v(x,t)| \varphi_0^R(x) \dd x\\
 \leq \int_{\R^d} |u_0(x)-v_0(x)| \varphi_R(x,t) \dd x \leq \int_{\R^d} |u_0(x)-v_0(x)| \sup_{s \in [0,T]} \varphi_R(x,s) \dd x,
\end{multline*}
for any $T \geq t \geq 0$. The right-hand side vanishes as $R \to \infty$ by the discussion above  and  the dominated convergence theorem. The proof of \eqref{cont-to-prove} is complete.
\end{proof}

 Here is  a regularization procedure for the weights. 

\begin{lemma}\label{lem:regularize-weight}
Assume \eqref{fassumption}, $m < M$, $\rho_\nu$ and $\theta_\nu$ are defined in \eqref{space-approx} and \eqref{time-approx}, and $0 \leq \varphi \in BLSC(\R^d \times \R^+)$ satisfies \eqref{nonlinear-dual-optimal} in Theorem \ref{thm:dual}. Then for any $\nu >0$, the convolution 
$$
\varphi_\nu:=\varphi \ast_{x,t} (\rho_\nu \theta_\nu) \in C_b^\infty(\R^d \times \R^+)
$$ 
also satisfies \eqref{nonlinear-dual-optimal} in Theorem \ref{thm:dual}. 
\end{lemma}

\begin{proof}
By assumption, 
\begin{equation}\label{tech-reg-weight}
\int_{\R^d} |u-v|(x,t) \varphi(x,s) \dd x \leq \int_{\R^d} |u_0-v_0|(x) \varphi(x,t+s)\dd x,
\end{equation}
for any $t,s \geq 0$, $u_0$ and $v_0$ with values in $[m,M]$, and entropy solutions $u$ and $v$ of \eqref{E} with $u_0$ and $v_0$ as initial data.
Our aim is to get the same result for $\varphi_\nu$. Let us use \eqref{tech-reg-weight} not for $u_0$ and $v_0$, but their translations $u_0(\cdot+y)$ and $v_0(\cdot+y)$ for some fixed $y \in \R^d$. Since the PDE part of \eqref{E} is invariant  w.r.t.  translation, the corresponding solutions are $u(x+y,t)$ and $v(x+y,t)$. Hence, 
$$
\int_{\R^d} |u-v|(x+y,t) \varphi(x,s) \dd x \leq \int_{\R^d} |u_0-v_0|(x+y) \varphi(x,t+s)\dd x
$$
for any $t,s \geq 0$. By changing the variable of integration, we obtain that 
$$
\int_{\R^d} |u-v|(x,t) \varphi(x-y,s) \dd x \leq \int_{\R^d} |u_0-v_0|(x) \varphi(x-y,t+s)\dd x.
$$
Now we fix $\tau \leq 0$ and apply this formula, not for $s$ but $s-\tau$. We deduce that 
\begin{equation}\label{last-referee}
  \int_{\R^d} |u-v|(x,t) \varphi(x-y,s-\tau) \dd x \leq \int_{\R^d} |u_0-v_0|(x) \varphi(x-y,t+s-\tau)\dd x.  
\end{equation}
Multiply then by $\rho_\nu(y) \theta_\nu(\tau)$ and integrate over $(y,\tau) \in \R^d \times \R^-$ to conclude. 
\end{proof}

 Later we will pass to the limit\footnote{This is the relaxed limit in \eqref{relaxed-referee} with the parameter $\nu$ instead of $\veps$.}
\begin{equation}\label{def-varphi-flat}
  \varphi_\flat:=\wliminf \varphi_\nu \quad \mbox{as } \nu \to 0^+,  
\end{equation}
 and compare $\varphi_\flat$ with the function $\varphi_\#$ defined in \eqref{precise-representative}. To compare the two limits, we  will  assume in addition that
\begin{equation}\label{hyp:support}
\textup{supp} (\rho_\nu) \subset B_\nu(0) \quad \mbox{and} \quad \textup{supp} (\theta_\nu) \subset (-\nu,0).
\end{equation}

 Here are fundamental properties of $\varphi_\flat$ and $\varphi_\#$ that will be needed. 

\begin{lemma}\label{lem:limit-reg}
Assume $\varphi \in BLSC(\R^d \times \R^+)$, $\varphi_\flat$ and $\varphi_\#$ are as above, and \eqref{hyp:support} holds.  Then: 
\begin{enumerate}[{\rm (i)}]
\item\label{item:tilde} 
The limit $\varphi_\flat$ is the pointwise largest function in $BLSC(\R^d \times \R^+)$ 
 that is  less than or equal 
$\varphi$ a.e. in $\R^d\times\R^+$.  Moreover $\varphi_\flat=\varphi$ a.e. in $\R^d\times\R^+$.

\smallskip

\item\label{item:star} For any $t \geq 0$, $\varphi_\#(\cdot,t)$ is the pointwise largest function in $BLSC(\R^d)$ 
less than or equal
$\varphi(\cdot,t)$ a.e. in $\R^d$.  Moreover $\varphi_\#(\cdot,t)=\varphi(\cdot,t)$ a.e. in $\R^d$.
\end{enumerate}
\end{lemma}

\begin{remark}
\begin{enumerate}[\rm (a)]
\item Above ``pointwise largest function'' means, e.g. for the item \eqref{item:tilde}, that if any other $\psi \in BLSC(\R^d \times \R^+)$ is such that $\psi \leq \varphi$ a.e. in $\R^d \times \R^+$, then necessarily 
$$
\psi(x,t) \leq \varphi_\flat(x,t) \quad \text{for all $(x,t) \in \R^d \times \R^+$.}
$$ 
The second item has to be understood similarly.
\item In the sequel, it is understood  that ``a.e.'' holds in $(x,t)$ in \eqref{item:tilde} and $x$ in \eqref{item:star}, without possibly recalling it.
\end{enumerate}
\end{remark}

\begin{proof} Let us prove \eqref{item:tilde}. Note first that $\varphi_\flat$ is lower semicontinuous as a lower relaxed limit. To prove that  $\varphi_\flat \leq \varphi$  a.e., it suffices to do it for the Lebesgue 
points
of $\varphi$. Such points $(x,t) \in \R^d \times (0,\infty)$ satisfy
$$
\lim_{\nu \to 0^+} \frac{1}{\nu^{d+1}} \iint_{B_\nu(x) \times (t-\nu,t+\nu)} |\varphi(y,s)-\varphi(x,t)| \dd y \dd s=0,
$$
so by the assumptions on the mollifiers, see \eqref{space-approx}, \eqref{time-approx} and \eqref{hyp:support}, we find that
\begin{align*}
|\varphi_\nu(x,t)-\varphi(x,t)|  & \leq 
\frac{1}{\nu^{d+1}} \iint_{B_\nu(x) \times (t,t+\nu)}  \left|\varphi(y,s)-\varphi(x,t)\right| \cdot \\
& \quad \cdot \rho \left(\frac{x-y}{\nu} \right) \theta\left(\frac{t-s}{\nu} \right) \dif y \dd s \to 0 \quad \text{as $\nu \to 0^+$.}
\end{align*}
It follows that
$$
\varphi_\flat(x,t) \leq \lim_{\nu \to 0^+} \varphi_\nu(x,t) =\varphi(x,t),
$$
at any Lebesgue point. 
Moreover, for any fixed $(x,t)$, lower semicontinuity of $\varphi$ implies that
$$
\varphi_\nu(y,s)
=\iint_{B_\nu(y) \times (s,s+\nu)} \underbrace{\varphi(z,\tau)}_{\geq \varphi(x,t)+\co(1)} \rho \left(\frac{y-z}{\nu} \right) \theta\left(\frac{s-\tau}{\nu} \right) \dif z \dd \tau \geq \varphi(x,t)+\co(1)
$$
as $(y,s,\nu) \to (x,t,0^+)$, and we 
 get 
that 
$$
\varphi_\flat(x,t)=\wliminf \varphi_\nu(x,t) \geq \varphi(x,t).
$$
We conclude that $\varphi_\flat=\varphi$ a.e.

Now, to  complete  the proof of \eqref{item:tilde}, it remains to prove that $\varphi_\flat \geq \psi$ pointwise for any other $\psi \in BLSC(\R^d \times \R^+)$ such that $\psi \leq \varphi$ a.e. Given such a function, let
$$
\psi_\flat:=\wliminf \psi \ast_{x,t} (\rho_\nu \theta_\nu).
$$
As above, $\psi \leq \psi_\flat$ pointwise;  but also $\psi_\flat \leq \varphi_\flat$ pointwise since
$$
\psi \ast_{x,t} (\rho_\nu \theta_\nu) \leq \varphi \ast_{x,t} (\rho_\nu \theta_\nu).
$$ 
This   proves \eqref{item:tilde} and the  arguments for \eqref{item:star} are similar.
\end{proof}

Here is  also  a general inequality between $\varphi_\flat$ and $\varphi_\#$  that will be needed. 

\begin{lemma}\label{lem:measurability}
Under the hypotheses of the previous lemma, $(\varphi_\#)_\ast \leq \varphi_\flat$ pointwise in $\R^d\times\R^+$.
\end{lemma}

\begin{proof}
Let us first prove that $\varphi_\#$ is measurable in $(x,t)$. We have
$$
\varphi_\#(x,t)=\sup_{n \geq 1} \overbrace{\inf_{m \geq n} \underbrace{\inf_{\substack{\frac{1}{m} \leq r \leq \frac{1}{n} \\ |y| \leq \frac{1}{n}}} \frac{1}{\textup{meas}(B_r(y))} \int_{B_r(y)} \varphi(x+z,t) \dd z}_{=:\varphi_{n,m}(x,t)}}^{=:\varphi_n(x,t)},
$$
where $n$ and $m$ are integers. For each $\frac{1}{m} \leq r \leq \frac{1}{n}$ and $|y| \leq \frac{1}{n}$, the function
$$
(x,t) \mapsto  \frac{1}{\textup{meas}(B_r(y))} \int_{B_r(y)} \varphi(x+z,t) \dd z
$$
is lower semicontinuous by Fatou's lemma and $\varphi \in BLSC$ (assumption in the previous  lemma). The infimum $\varphi_{n,m}$ remains lower semicontinuous, because $r$ and $y$ live in compact sets. Hence, 
$\varphi_n=\inf_{m \geq n} \varphi_{n,m}$ is measurable in $(x,t)$ and so is 
$\varphi_\#=\sup_{n \geq 1} \varphi_n$. 

We can now prove the lemma. For any $t \geq 0$, the measurable functions $\varphi,\varphi_\#$ satisfy $\varphi_\#(\cdot,t)=\varphi(\cdot,t)$ a.e., hence we may use the Fubini theorem to conclude that
$$
\iint_{\R^d \times \R^+} \mathbf{1}_{\{\varphi_\#=\varphi\}} \dd x \dd t=\int_{\R^+} \left( \int_{\R^d} \mathbf{1}_{\{\varphi_\#(x,t)=\varphi(x,t)\}}  \dd x \right) \dif t=0.
$$
This proves that $\varphi_\#=\varphi$ a.e. in $(x,t)$, so that $(\varphi_\#)_\ast \leq \varphi$ a.e. in $(x,t)$. Hence $(\varphi_\#)_\ast \leq \varphi_\flat$ pointwise by Lemma \ref{lem:limit-reg}\eqref{item:tilde}. 
\end{proof}

Here are further properties that we will need.

\begin{lemma}\label{lem:starmult-ineq}
Let $\varphi,\psi\in BLSC(\R^d\times\R^+)$ and $\varphi_\#,\psi_\#$ as in \eqref{precise-representative}. Then
\begin{enumerate}[{\rm (i)}]
\item\label{missed-label} $\varphi\leq (\varphi_\#)_\ast$ pointwise, and
\item\label{missed-label-bis} if $\varphi\leq \psi_\#$ pointwise, then $\varphi_\#\leq \psi_\#$ pointwise.
\end{enumerate}
\end{lemma}

\begin{proof}
We can show that $\varphi\leq \varphi_\#$ from the definition of $\varphi_\#$ and the lower semicontinuity of $\varphi$, exacty as we showed that $\varphi \leq \varphi_\flat$ in the proof of Lemma \ref{lem:limit-reg}. In particular, $\varphi\leq (\varphi_\#)_\ast$ which is part \eqref{missed-label}. For part \eqref{missed-label-bis}, use Lemma \ref{lem:limit-reg}\eqref{item:star}. It says that $\psi_\#(\cdot,t)=\psi(\cdot,t)$ a.e. in $x$, for each fixed $t \geq 0$. Hence, $\varphi(\cdot,t) \leq \psi(\cdot,t)$ a.e. and the desired inequality follows again from the definitions of $\varphi_\#$ and $\psi_\#$.
\end{proof}

We are now in position to prove  Theorem \ref{thm:dual}. 

\begin{proof}[Proof of Theorem \ref{thm:dual}]

Let us proceed in several steps.

\medskip

\noindent {\bf 1)} \eqref{item:dual}$\implies$\eqref{nonlinear-dual-optimal}.

\smallskip

\noindent By \eqref{item:dual}, $(\varphi_\#)_\ast$ is a $BLSC$ supersolution of \eqref{eq:dual-true-1}. In particular, for any fixed $s \geq 0$, the function
$$
(x,t) \mapsto (\varphi_\#)_\ast(x,t+s)
$$
is also a supersolution of \eqref{eq:dual-true-1}. By Remark \ref{rem:weighted}\eqref{weighted-supersol},  we can apply Theorem \ref{thm:weighted} to this supersolution with the $BLSC$ initial weight $(\varphi_\#)_\ast(\cdot,s)$. The result is that
\begin{equation*}
\int_{\R^d} |u-v|(x,t) (\varphi_\#)_\ast(x,s) \dd x \leq \int_{\R^d} |u_0-v_0|(x) (\varphi_\#)_\ast(x,t+s)\dd x,
\end{equation*}
for any $u_0=u_0(x)$ and $v_0=v_0(x)$ with values in $[m,M]$, $u$ and $v$ entropy solutions of \eqref{E} with $u_0$ and $v_0$ as initial data, and $t,s \geq 0$. This is exactly \eqref{nonlinear-dual-optimal} but with $(\varphi_\#)_\ast$ instead of $\varphi$.  To replace $(\varphi_\#)_\ast$ by $\varphi$, we use Lemma \ref{lem:starmult-ineq}\eqref{missed-label} for the left-hand side. For the right-hand side, we use that 
$(\varphi_\#)_\ast \leq \varphi_\#$ pointwise and the fact that $\varphi_\#(x,t+s)=\varphi(x,t+s)$ for  a.e.  $x$, see Lemma \ref{lem:limit-reg}\eqref{item:star}. This implies \eqref{nonlinear-dual-optimal} with $\varphi$, as desired.

\medskip

\noindent {\bf 2)} \eqref{nonlinear-dual-optimal}$\implies$\eqref{item:dual} {\it for smooth weights $\varphi$.}

\smallskip

\noindent Let us prove the reverse implication when $0\leq\varphi \in C^\infty_b(\R^d \times \R^+)$. We will appropriately choose $u_0$ and $v_0$ later. For the moment, we assume that
$$
m \leq v_0 \leq u_0 \leq M \quad \mbox{and} \quad u_0-v_0 \in L^1(\R^d). 
$$
By Lemmas \ref{lem:max-comp} and \ref{lem:L1}, $0 \leq u-v \in C(\R^+;L^1(\R^d))$, and then we can use \eqref{nonlinear-dual-optimal} to 
get   
\begin{equation}\label{tech-dual-first}
\int_{\R^d} (u-v)(x,T) \varphi(x,s) \dd x \leq \int_{\R^d} (u_0-v_0)(x) \varphi(x,T+s)\dd x,
\end{equation}
for any $T,s \geq 0$. Let us fix $s >0$ and determine what PDE $\varphi$ satisfies. This will be done by injecting the weak formulation of \eqref{E} into \eqref{tech-dual-first} and then pass to the limit as $T \to 0^+$. By Lemma \ref{lem:distribution},
\begin{equation*}
\begin{split}
\int_{\R^d} (u-v)(x,T) \phi(x,T) \dd x & = \iint_{\R^d \times (0,T)} \left((u-v)\partial_t\phi+\sum_{i=1}^d (F_i(u)-F_i(v))\partial_{x_i}\phi\right.\\
& \quad \left.+\sum_{i,j=1}^{d} (\mathcal{A}_{ij}(u)-\mathcal{A}_{ij}(v))\partial_{x_ix_j}^2\phi\right)\dif x\dd t\\
&\quad +\int_{\R^d} (u_0-v_0)(x) \phi(x,0) \dd x,
\end{split}
\end{equation*}
for any $\phi \in C^\infty_c(\R^d \times [0,T])$ and $\mathcal{A}_{ij}'=A_{ij}$. Note that we have rewritten the equation given by Lemma \ref{lem:distribution} with integrals in $t<T$ and an additional final term at $t=T$. This follows from standard arguments using the $L^1_{\textup{loc}}$ continuity in time of $u$ and $v$.
Since $\varphi \in C^\infty_b$, $u-v \in C_t(L^1_x)$ and $u,v \in L^\infty$, a 
standard 
approximation 
argument
shows that we can take $\phi$ to be
$$
\phi(x,t)=\varphi(x,t+s-T),
$$
and get that
\begin{equation}\label{tech-dual-second}
\begin{split}
\int_{\R^d} (u-v)(x,T) \varphi(x,s) \dd x
&= \iint_{\R^d \times (0,T)} \Bigg((u-v)\partial_t\varphi(t+s-T)\\
& \quad +\sum_{i=1}^d (F_i(u)-F_i(v))\partial_{x_i}\varphi(t+s-T)\\
& \quad \left.+\sum_{i,j=1}^{d} (\mathcal{A}_{ij}(u)-\mathcal{A}_{ij}(v))\partial_{x_ix_j}^2\varphi(t+s-T)\right)\dif x\dd t\\
&\quad +\int_{\R^d} (u_0-v_0)(x) \varphi(x,s-T) \dd x.
\end{split}
\end{equation}
Here we assume that $s>0$ and $T$ is so small that $s-T>0$.
Inserting 
\eqref{tech-dual-second} into \eqref{tech-dual-first}, we get
\begin{align*}
& \int_{\R^d} (u_0-v_0)(x) \varphi(x,s+T)\dd x \\
& -\int_{\R^d} (u_0-v_0)(x) \varphi(x,s-T) \dd x  \geq \iint_{\R^d \times (0,T)} \Big( \dots \Big) \dd x \dd t.
\end{align*}
We now would like to divide by $2T$ and pass to the limit as $T \to 0^+$. All the computations are justified, again because $\varphi \in C^\infty_b$, the solutions $u$ and $v$ are bounded, and  $u-v \in C_t(L^1_x)$.  We get that
\begin{equation*}
\begin{split}
& \int_{\R^d} (u_0(x)-v_0(x)) \partial_s \varphi(x,s) \dd x \\
& \geq \frac{1}{2} \int_{\R^d} \left((u_0-v_0)\partial_s \varphi(s)+\sum_{i=1}^d (F_i(u_0)-F_i(v_0))\partial_{x_i}\varphi(s) \right. \\
& \quad \left.+\sum_{i,j=1}^{d} (\mathcal{A}_{ij}(u_0)-\mathcal{A}_{ij}(v_0))\partial_{x_ix_j}^2\varphi(s)\right)\dif x.
\end{split}
\end{equation*}
Substracting the term $\int (u_0-v_0)\partial_s \varphi(s) \dd x/2$ of the right-hand side implies that
\begin{equation}\label{tech-dual-third}
\begin{split}
& \int_{\R^d} (u_0(x)-v_0(x)) \partial_s \varphi(x,s) \dd x \\
& \geq \int_{\R^d} \left(\sum_{i=1}^d (F_i(u_0)-F_i(v_0))\partial_{x_i}\varphi(s)+\sum_{i,j=1}^{d} (\mathcal{A}_{ij}(u_0)-\mathcal{A}_{ij}(v_0))\partial_{x_ix_j}^2\varphi(s)\right)\dif x\\
&= \int_{\R^d} \int_{v_0(x)}^{u_0(x)} \left\{F'(\xi) \cdot D \varphi(x,s)+\tr \left(A_{ij}(\xi)  D^2 \varphi(x,s) \right) \right\} \dif \xi \dd x,
\end{split}
\end{equation}
for any $s >0$ and $0 \leq u_0-v_0 \in L^1(\R^d)$ such that both $u_0$ and $v_0$ take their values in the interval $[m,M]$. It remains to choose $u_0-v_0$ as an approximate unit, up to some multiplicative constant.

Let us introduce new parameters: $x_0 \in \R^d$, $\varepsilon>0$ and $m \leq a < b \leq M$. We would like to choose 
\begin{equation}\label{eq:approxunitmultiplicativeconstant}
u_0-v_0=(b-a) \mathbf{1}_{x_0+(-\varepsilon,\varepsilon)^d},
\end{equation}
with the constraint that both $u_0$ and $v_0$ only take the two values $a$ and $b$. Writing $x=(x_i)$, take e.g.  
$$
u_0(x):=
\begin{cases}
a & \text{if } x_1 >(x_0)_1+\varepsilon,\\
b & \text{ if not},
\end{cases}
$$
and 
$$
v_0(x):=
\begin{cases}
a & \text{if } x_1 >(x_0)_1+\varepsilon \mbox{ or } x\in x_0+(-\veps,\veps)^d,\\
b & \text{if not}.
\end{cases}
$$
Then $m \leq v_0 \leq u_0 \leq M$ and $u_0-v_0 \in L^1(\R^d)$ as required.  
Inserting
our choice into \eqref{tech-dual-third} and dividing by $(b-a) \varepsilon^d $, we deduce that
\begin{equation*}
\begin{split}
& \frac{1}{\varepsilon^d} \int_{x_0+(-\varepsilon,\varepsilon)^d} \partial_s \varphi(x,s) \dd x \\
& \geq \frac{1}{\varepsilon^d} \int_{x_0+(-\varepsilon,\varepsilon)^d}  \frac{1}{b-a} \int_{a}^{b} \left\{F'(\xi) \cdot D \varphi(x,s)+\tr \left(A_{ij}(\xi)  D^2 \varphi(x,s) \right) \right\} \dif \xi \dd x.
\end{split}
\end{equation*}
Let now $\xi \in (m,M)$ be any Lebesgue point of any arbitrarily chosen a.e. representative of $(F',A)$. Take first the limit as $a,b \to \xi$ such that
$\xi$ is the center of each $[a,b]$ in order to use the Lebesgue point property; take next the limit as $\varepsilon \to 0^+$. This gives us that
$$
\partial_s \varphi(x_0,s) \geq F'(\xi) \cdot D \varphi(x_0,s)+\tr \left(A_{ij}(\xi)  D^2 \varphi(x_0,s) \right),
$$
for any $x_0 \in \R^d$, $s>0$, and Lebesgue point $\xi$. That is $\varphi$ is a supersolution of \eqref{eq:dual-true}. This completes the proof of the remaining implication in the case where $\varphi$ is $C^\infty_b$ (and then $\varphi_\#=\varphi$).

\medskip

\noindent {\bf 3)} \eqref{nonlinear-dual-optimal}$\implies$\eqref{item:dual} {\it for nonnegative BLSC weights $\varphi$.}

\smallskip

\noindent In this case we use the regularization procedure of Lemma \ref{lem:regularize-weight}. By this lemma
$$
\varphi_\nu=\varphi \ast_{x,t} (\rho_\nu \theta_\nu)
$$ 
satisfies \eqref{nonlinear-dual-optimal} since $\varphi$ does by assumption. By the previous step we deduce that $\varphi_\nu$ is a supersolution of \eqref{eq:dual-true-1}. Hence
$$
\varphi_\flat=\wliminf \varphi_\nu 
$$
is also a supersolution by stability (cf. Proposition \ref{prop:rsl}). But to prove \eqref{item:dual}, we need to show that $\varphi_\#$ is a supersolution. We will do this by showing that $\varphi_\flat=(\varphi_\#)_\ast$ pointwise (at least for positive times). 
To prove that $\varphi_\flat \leq (\varphi_\#)_\ast$, we need to use \eqref{nonlinear-dual-optimal}. By \eqref{nonlinear-dual-optimal},
\begin{equation*}
\int_{\R^d} |u-v|(x,t) \varphi(x,s) \dd x \leq \int_{\R^d} |u_0-v_0|(x) \varphi(x,t+s) \dd x,
\end{equation*}
for any $u_0$ and $v_0$ in $L^\infty(\R^d,[m,M])$ and corresponding solutions $u$ and $v$ of \eqref{E} and $t,s \geq 0$.
By Lemma \ref{lem:limit-reg}\eqref{item:tilde}, we also have that $\varphi_\flat=\varphi$ a.e. In particular, there is a null set $N \subset \R^+$ such that $\varphi(\cdot,s)=\varphi_\flat(\cdot,s)$ a.e., for any $s \notin N$.\footnote{To find $N$ use that $\iint_{\R^d \times \R^+} \mathbf{1}_{\{\varphi_\flat=\varphi\}} \dd x \dd s=0=\int_{\R^+} \textup{meas} \{\varphi(\cdot,s)=\varphi_\flat(\cdot,s)\} \dd s$ by Fubini.} Fixing $T > 0$, there thus exists a sequence $s_n \to T^-$ such that $s_n \notin N$, for any $n$. Choosing moreover $t_n:=T-s_n$, we deduce that 
 \begin{equation*}
\int_{\R^d} |u-v|(x,t_n) \varphi_\flat(x,s_n) \dd x \leq \int_{\R^d} |u_0-v_0|(x) \varphi(x,T) \dd x.
\end{equation*}
Let us pass to the limit as $n \to \infty$ in the left-hand side. To do so, we use Fatou's lemma, which is possible because of the lower semicontinuity of $\varphi_\flat$ and the continuity of entropy solutions with values in $L^1_{\textup{loc}}(\R^d)$ which implies that
$$
|u-v|(x,t_n) \to |u_0-v_0|(x) \quad \mbox{for a.e. $x$}
$$
(along a subsequence). In the limit, it then follows that  
 \begin{equation*}
\int_{\R^d} |u_0-v_0|(x) \varphi_\flat(x,T) \dd x \leq \int_{\R^d} |u_0-v_0|(x) \varphi(x,T) \dd x
\end{equation*}
for any $u_0$ and $v_0$ in $L^\infty(\R^d,[m,M])$ and $T>0$. To continue, we argue as in the previous step where we chose $0 \leq u_0-v_0 \in L^1(\R^d)$ to be an approximate unit up to a multiplicative constant, cf. \eqref{eq:approxunitmultiplicativeconstant}. The same arguments imply that for any $T>0$, 
$$
\varphi_\flat(\cdot,T) \leq  \varphi(\cdot,T) \quad \mbox{a.e.}
$$
By Lemma \ref{lem:limit-reg}\eqref{item:star}, we conclude that $\varphi_\flat \leq \varphi_\#$ pointwise (for positive times). Hence, $\varphi_\flat \leq (\varphi_\#)_\ast$ and then $\varphi_\flat = (\varphi_\#)_\ast$ pointwise (for positive times) by Lemma \ref{lem:measurability}. This implies that $(\varphi_\#)_\ast=\varphi_\flat$ is a supersolution of \eqref{eq:dual-true-1}. The proof of Theorem \ref{thm:dual} is complete. 
\end{proof}

We have  now established all preliminary results and are ready to prove  our duality results (Corollaries \ref{cor:opt} and \ref{cor:sg}). 

\begin{proof}[Proof of Corollary \ref{cor:opt}]
We already know that $\underline{\varphi} \in \mathscr{W}_{m,M,\varphi_0}$ by Theorem \ref{thm:weighted}. 
Let us prove the formula with the inf. Take $\varphi \in \mathscr{W}_{m,M,\varphi_0}$, which means that $\varphi \in BLSC$ and satisfies Theorem \ref{thm:dual}\eqref{nonlinear-dual-optimal} with $\varphi(t=0) \geq \varphi_0$. By this theorem, $\varphi$ satisfies \eqref{item:dual} as well, that is $\varphi_\#$ is a supersolution of \eqref{eq:dual-true-1}. Recall that $\varphi \leq (\varphi_\#)_\ast$ pointwise by Lemma \ref{lem:starmult-ineq}\eqref{missed-label}. In particular
$$
(\varphi_\#)_\ast(t=0) \geq \varphi(t=0) \geq \varphi_0.
$$
Thus $\varphi_\#$ is a supersolution of the Cauchy problem \eqref{eq:dual-true}, and $\underline{\varphi} \leq \varphi_\#$ by  Proposition \ref{prop:extremal-reformulation}.  Then Lemma \ref{lem:starmult-ineq}\eqref{missed-label-bis} implies that $(\underline{\varphi})_\#\leq\varphi_\#$ pointwise, and we conclude that
$$
(\underline{\varphi})_\#(x,t) = \inf \left\{\varphi_\#(x,t):\varphi \in  \mathscr{W}_{m,M,\varphi_0}\right\} \quad \forall (x,t) \in \R^d \times \R^+
$$
(with an equality because $\underline{\varphi} \in \mathscr{W}_{m,M,\varphi_0}$). The proof is complete.
\end{proof}

\begin{proof}[Proof of Corollary \ref{cor:sg}]
 Fix $m<M$.  By what precedes, the solution semigroup $G_t$ of \eqref{eq:dual-true} is  a  strongly continuous  semigroup of continuous operators  on $C_b \cap L^\infty_\INT(\R^d)$ and satisfies \eqref{etoile-bis}. Let now $H_t$ be another arbitrary such semigroup satisfying \eqref{etoile-bis}, i.e. such that
\begin{equation*}
\int_{\R^d} |S_tu_0-S_tv_0| \varphi_0 \dd x \leq \int_{\R^d} |u_0-v_0| H_t \varphi_0 \dd x,
\end{equation*}
for any $u_0$ and $v_0$ in $L^\infty(\R^d,[m,M])$, $0 \leq \varphi_0\in C_{b} \cap L^\infty_{\INT}(\R^d)$, and $t \geq 0$. We have to prove that for any such $\varphi_0$ and $t$,
$$
G_t \varphi_0 \leq H_t \varphi_0.
$$
First  the minimal solution of \eqref{eq:dual-true} is the unique continuous solution, that is 
$$
\underline{\varphi}(x,t)=G_t  \varphi_0(x) \quad \forall (x,t) \in \R^d \times \R^+. 
$$
Moreover, the above assumption on $H_t$ implies that
$$
H_t \varphi_0(x) \in \mathscr{W}_{m,M,\varphi_0}.
$$
By Corollary \ref{cor:opt} we deduce that for any $x \in \R^d$ and $t \geq 0$, 
$$
\left(G_t  \varphi_0\right)_\#(x) \leq \left(H_t  \varphi_0\right)_\#(x),
$$
where we recall that
$$
\left(G_t  \varphi_0\right)_\#(x)=\liminf_{\substack{r \to 0^+\\y \to x}} \frac{1}{\textup{meas}(B_r(y))} \int_{B_r(y)} G_t\varphi_0(z) \dd z
$$
(and 
similarly
for $H$). Since both $G_t  \varphi_0(x)$ and $H_t  \varphi_0(x)$ are continuous in $x$, we have $(G_t  \varphi_0)_\#=G_t  \varphi_0$ and $(H_t  \varphi_0)_\#=H_t  \varphi_0$ pointwise and the proof is complete. 
\end{proof}

\subsection{$L^\infty_\INT$ versus $L^1$: Proofs of Propositions \ref{natural-second}, \ref{lem:blow-up},  and Theorem \ref{thm:qualitative-HJB}}
\label{ref-title-precise-instability}

Recall  that these results  justify the use of $L^\infty_\INT$ for \eqref{eq:dual}, instead of $L^1$.  We need a result on the  profile $U(r)=c_0 \int_r^\infty \e^{-\frac{s^2}{4}} \dd s$ with $c_0$ such that $U(0)=1$. 

\begin{lemma}\label{tech-lem-example}
For any $(x,t) \in \R \times \R^+$, let
\begin{equation*}
\psi(x,t) :=
\begin{cases}
 U \left(|x|/\sqrt{t}\right) & \mbox{if } t>0,\\
 \mathbf{1}_{\{0\}}(x) & \mbox{if } t=0.
\end{cases}
\end{equation*}
 Then  $\psi \in BUSC(\R \times \R^+)$ and is a subsolution of \eqref{example-diffusion}. 
\end{lemma}

\begin{proof}
Let us prove that $\psi$ is a subsolution of \eqref{example-diffusion}. In the domain
$
\{x \neq 0,t>0\}, 
$
we  find as in the proof of Lemma \ref{esti-with-decay} that 
$$
\partial_t \psi=\partial_{xx}^2 \psi=(\partial_{xx}^2 \psi)^+
$$
in the classical  sense.  If now $x=0$, we have 
$$
\partial_t \psi (0,\cdot)=0 \leq (\partial_{xx}^2 \psi(0,\cdot))^+
$$ 
since $\psi(0,\cdot)$ is constant in time. Let us now show that $\psi$ is $BUSC$. It is clearly continuous for positive $t$ and it only remains to prove that 
\begin{equation*}
\mathbf{1}_{x=0} \geq \limsup_{\R \times \R^+ 
 \ni 
(y,t)\to (x,0)} U \left(|y|/\sqrt{t}\right),
\end{equation*}
for any $x \in \R$. If $x =0$, the result follows since $U(r) \leq U(0) = 1$ for any $r \geq 0$. If  $x \neq 0$, then we use that 
$$
|y|/\sqrt{t} \to \infty \quad \mbox{as } (y,t) \to (x,0^+)
$$  
together with the fact that $\lim_{r \to \infty} U(r)=0$. The proof of Lemma \ref{tech-lem-example} is now complete.
\end{proof}

\begin{proof}[Proof of Proposition \ref{natural-second}] 
Theorem \ref{thm:quantitative-HJB-unif} implies the if-part. Let us prove  the only-if-part.  
It is based on the following pointwise lower bound:
\begin{equation}\label{claim-ns}
\varphi(x,t) \geq U\left(1/\sqrt{t}\right) \sup_{x+[-1,1]} \varphi_0 \quad \forall x \in \R, \forall t > 0,
\end{equation}
where $U$ is the profile from the previous lemma, 
 $0\leq \varphi_0\in C_b(\R)$  and $\varphi$ is the solution of \eqref{example-diffusion} with $\varphi_0$ as initial data. 
Let us prove \eqref{claim-ns}. 
Fix $x$ and $t$. 
The sup on the right-hand side is attained at some $x_0 \in x+[-1,1]$. By the previous lemma, 
$$
(y,s) \mapsto \varphi_0(x_0) U \left(|y-x_0|/\sqrt{s}\right)
$$
is a $BUSC$ subsolution of \eqref{example-diffusion}. At $s=0$, it equals the function
$$
y \mapsto \varphi_0(x_0) \mathbf{1}_{\{x_0\}}(y)
$$ 
which is less or equal to $\varphi_0=\varphi_0(y)$. 
By the comparison principle (Theorem \ref{propofvissoln}),
$$
\varphi(y,s) \geq \varphi_0(x_0) U \left(|y-x_0|/\sqrt{s}\right) \quad \forall y \in \R, \forall s >0.
$$
Taking $(y,s)=(x,t)$, we then get that
$$
\varphi(x,t) \geq \underbrace{\varphi_0(x_0)}_{=\,\sup_{x+[-1,1]} \varphi_0} \underbrace{U \left(|x-x_0|/\sqrt{t}\right)}_{\geq \,U(1/\sqrt{t})}.
$$
This completes the proof of \eqref{claim-ns}. From that bound the only-if-part of Proposition \ref{natural-second} is obvious since $U(1/\sqrt{t})$ is positive for $t>0$.
\end{proof}

\begin{proof}[Proof of Proposition \ref{lem:blow-up}] Let $x_0 \in \R$ and $c>0$ be such that
$$
\rho \geq c \mathbf{1}_{\{x_0 \}},
$$
where $\rho$ is defined in \eqref{kernel-blow-up}, and define 
$$
\psi_n(x,t):=nc \psi \left(nx-x_0,n^2 t \right),
$$ 
where $\psi$ is given by Lemma \ref{tech-lem-example}. It is easy to see that $\psi_n$ remains a subsolution of \eqref{example-diffusion}. Moreover, it is $BUSC$ with 
$$
\varphi_n(x,0) \geq \psi_n(x,0) \quad \forall x \in \R,
$$ 
by \eqref{kernel-blow-up}. Hence $\varphi_n \geq \psi_n$ by the comparison principle and  it suffices to show that 
\begin{equation*}
\lim_{n \to \infty} \psi_n(x,t)=\infty \quad \forall x \in \R, \forall t >0.
\end{equation*}
But this is quite easy because 
$$
\psi_n(x,t)=nc U \left(\left|x-\frac{x_0}{n}\right|/\sqrt{t}\right),
$$
for any $x \in \R$ and $t>0$, and both the constant $c$ and the profile $U(\cdot)$ are positive. The proof of Proposition \ref{lem:blow-up} is complete.
\end{proof}

To show Theorem \ref{thm:qualitative-HJB}, we need the following lemma whose proof is elementary and left to the reader.

\begin{lemma}\label{lem:basic-coarsest}
For any $\varphi_0:\R^d \to \R^d$, $\sup |\varphi_0| \leq |\sup \varphi_0| + |\inf \varphi_0|$.
\end{lemma}

\begin{proof}[Proof of Theorem \ref{thm:qualitative-HJB}]
The fact that $E=C_{b} \cap L^\infty_{\INT}(\R^d)$ satisfies \eqref{coarsest-space}--\eqref{coarsest-semigroup} follows from Theorem \ref{thm:uniformly-integrable}  and Corollary \ref{cor-nv}. 
Let now $E$ be another normed space satisfying such properties and let us prove that it is continuously  embedded  into $C_{b} \cap L^\infty_{\INT}(\R^d)$.  Recall  that \eqref{coarsest-semigroup} is required to hold for any data $b=b(\xi)$ and $a=a(\xi)$ satisfying \eqref{bassumption}. Choose e.g.  the eikonal equation
\begin{equation*}
\partial_t \varphi=\sum_{i=1}^d |\partial_{x_{i}} \varphi|
\end{equation*}
and denote by $G^e_t$ its semigroup. By the representation Proposition \ref{representation-deterministic}, 
\begin{equation*}
G^e_t \varphi_0(x)=\sup_{x+t [-1,1]^d} \varphi_0.
\end{equation*}
Since $G^e_{t=1}$ maps  $E \subseteq C_b \cap L^1(\R^d)$ into $X=\overline{E}^{\|\cdot\|_E} \subseteq L^1(\R^d)$ by assumption, the function
$$
x \mapsto \sup_{x+[-1,1]^d} \varphi_0
$$
belongs to $L^1(\R^d)$ for any $\varphi_0 \in E$. Using that $E$ is a vector space, $-\varphi_0 \in E$, and the function
$$
x \mapsto \inf_{x+[-1,1]^d} \varphi_0
$$
also belongs to $L^1(\R^d)$. By Lemma \ref{lem:basic-coarsest}, we conclude that $E \subseteq C_{b} \cap L^\infty_{\INT}(\R^d)$. Finally we use that   
$
G^e_{t=1}:E \to X 
$ 
is continuous at $\varphi_0 \equiv 0$ to obtain that for any $\|\varphi_0^n\|_E \to 0$,   as $n \to \infty$, we have  
$
G^e_{t=1} \varphi_0^n \to 0
$ in $X$. 
Combining this with the continuity of the  inclusion $X 
\subseteq L^1(\R^d)$,  we obtain that
$$
\left\|\sup_{x+[-1,1]^d} \varphi_0^n \right\|_{L_x^1} \to 0.
$$
Using once again that $E$ is a normed space, the same holds with $-\varphi_0$, that is
$$
\left\|\inf_{x+[-1,1]^d} \varphi_0^n \right\|_{L_x^1} \to 0.
$$
By Lemma \ref{lem:basic-coarsest}, we conclude that  $\|\varphi_0^n\|_{\INT} \to 0$ which completes the proof. 
\end{proof}

\section*{Acknowledgements}
 
E. R. Jakobsen is supported by the Research Council of Norway under grant agreement no. 325114 ``IMod. Partial differential equations, statistics and data: An interdisciplinary approach to data-based modelling.'' J. Endal received funding from the Research Council of Norway under the Toppforsk (research excellence) grant agreement no. 250070 ``Waves and Nonlinear Phenomena (WaNP),'' the European Union’s Horizon 2020 research and innovation programme under the Marie Sk{\l}odowska-Curie grant agreement no. 839749 ``Novel techniques for quantitative behavior of convection-diffusion equations (techFRONT),'' and from the Research Council of Norway under the MSCA-TOPP-UT grant agreement no. 312021.   N. Alibaud was supported by the French ANR project CoToCoLa (no. ANR-11-JS01-006-01). Part of this paper was completed during J. Endal's stay at Laboratory of Mathematics of Besan\c{c}on (LmB), Universit\'e de Franche-Comt\'e, and SUPMICROTECH-ENSMM. Finally, we thank Boris Andreianov for  comments and the anonymous referees for detailed reports and suggestions for the paper.


\appendix

\section{Complementary proofs for viscosity solutions}
\label{app:viscosity}

\subsection{Minimal viscosity solutions}\label{app:min-max}

Here are the proofs of Theorem \ref{thm:propextremalsoln1} and Propositions \ref{prop:extremal-reformulation} and \ref{stability-extremal}; the ideas are inspired by \cite{CrLiSo89,BaSoSo93,GiSa01} and the details are given for completeness.

\begin{proof}[Proof of Theorem \ref{thm:propextremalsoln1}]
Consider the infconvolution $(\varphi_0)_\varepsilon$ as 
in \eqref{infconvol}, 
which is at least $C_b$ with
$
\inf \varphi_0 \leq (\varphi_0)_\varepsilon \leq (\varphi_0)_\ast \leq \sup \varphi_0,
$ and
$$
\lim_{\varepsilon \downarrow 0} \uparrow (\varphi_0)_\varepsilon=\sup_{\varepsilon >0} (\varphi_0)_\varepsilon=(\varphi_0)_\ast,
$$
see e.g. \cite{CrIsLi92,FlSo93,Bar94,BaCa-Do97}. Let $\varphi_\varepsilon$ be the viscosity solution of \eqref{eq:dual-1} with initial data $(\varphi_0)_\varepsilon$, whose well-posedness is ensured by Theorem \ref{exisuniqvissoln}. By the maximum principle, see Remark \ref{rem:max}, we have the bounds
$$
\inf \varphi_0 \leq \varphi_\varepsilon \leq \sup \varphi_0.
$$
We can then define the real-valued and bounded function  
$$
\underline{\varphi}:=\sup_{\varepsilon>0} \varphi_\varepsilon.
$$
We will see that this is our desired minimal solution. 

The key step is to prove that 
\begin{equation}\label{formula:extremal}
\underline{\varphi}=\sup_{\varepsilon>0} \varphi_\varepsilon=\wliminf \varphi_\varepsilon 
\end{equation}
where the relaxed limit is taken as $\varepsilon \to 0^+$.   This follows by elementary arguments (see e.g. \cite{Bar94,BaCa-Do97}) since $\varphi^\varepsilon$ is at least lower semicontinuous and  nondecreases as $\varepsilon \downarrow 0$, which follows by comparison since $(\varphi_0)^\varepsilon$ nondecreases as $\varepsilon \downarrow 0$. Let us give details for the reader's convenience.
For any fixed $(x,t)$, 
$$
\wliminf \varphi_\varepsilon(x,t) \leq \lim_{\varepsilon \to 0^+} \varphi_\varepsilon(x,t) =\underline{\varphi}(x,t).
$$
Moreover, for any sequence $(x_n,t_n,\varepsilon_n) \to (x,t,0^+)$ such that $\varepsilon_n \leq \varepsilon_m$ for any $n \geq m$, we have 
$
\varphi_{\varepsilon_n}(x_n,t_n) \geq  \varphi_{\varepsilon_m}(x_n,t_n). 
$
Fixing $m$ and taking the limit in $n$, 
$$
\liminf_{n \to \infty} \varphi_{\varepsilon_n}(x_n,t_n) \geq \liminf_{n \to \infty} \varphi_{\varepsilon_m}(x_n,t_n) \geq \varphi_{\varepsilon_m}(x,t)
$$
by lower semicontinuity  of $\varphi_{\varepsilon_m}$. Taking the limit in $m$,
$$
\liminf_{n \to \infty} \varphi_{\varepsilon_n}(x_n,t_n) \geq \lim_{m \to \infty} \varphi_{\varepsilon_m}(x,t)=\underline{\varphi}(x,t).
$$
This proves \eqref{formula:extremal}. 

By stability by sup (Proposition \ref{prop:ssi}), $\underline{\varphi}$ is a subsolution of \eqref{eq:dual-1}, and by stability by relaxed limit (Proposition \ref{prop:rsl}), $\underline{\varphi}$ is a supersolution of \eqref{eq:dual-1}. To pass to the limit in the initial data, use Proposition \ref{prop:rsl-bis} to infer that
$$
(\underline{\varphi})^\ast(x,t=0) \leq \wlimsup \varphi_\veps(x,0) = \wlimsup \left[\varphi_\veps (\cdot,0) \right](x) \leq (\varphi_0)^\ast(x)
$$
(the first relaxed limit as $\veps\to0^+$ is in $(x,t)$ and the second in $x$). This gives the  inequality of subsolution as in Definition \ref{def:viscsoln}\eqref{ref:id}. For the other inequality, use that $\underline{\varphi}$ is lower semicontinuous, as a sup of continuous functions, with
\begin{equation*}
\underline{\varphi}(x,t=0)=\sup_{\varepsilon >0} \varphi_\varepsilon(x,0)=(\varphi_0)_\ast(x).
\end{equation*} 
This proves that $\underline{\varphi}$ is a solution of \eqref{eq:dual}. It only remains to prove that it is minimal. Let $\varphi$ be another bounded discontinuous solution. Noting that
$$
(\varphi_0)_\varepsilon \leq (\varphi_0)_\ast \leq \varphi_\ast(t=0),
$$
we use once more the comparison principle to deduce that
$
\varphi_\varepsilon \leq \varphi,
$
for any $\varepsilon >0$, so $\underline{\varphi} \leq \varphi$ as $\varepsilon \to 0^+$.
\end{proof}

\begin{proof}[Proof of Proposition \ref{prop:extremal-reformulation}]
We argue as in the end of the proof of Theorem \ref{thm:propextremalsoln1}: Assume $\varphi$ is a bounded supersolution of \eqref{eq:dual}, then
$
(\varphi_0)_\varepsilon \leq (\varphi_0)_\ast \leq \varphi_\ast(t=0)
$
and, by comparison, $
\varphi_\varepsilon \leq \varphi$, etc. 
\end{proof}

\begin{proof}[Proof of Proposition \ref{stability-extremal}]
Let $\underline{\varphi}$ denote the minimal solution of \eqref{eq:dual} with initial data  $\varphi_0:=\sup_{n} (\varphi_0^n)_\ast$.  We have to prove that $\underline{\varphi}=\sup_n \underline{\varphi}_n$, where $\underline{\varphi}_n$ is the minimal solution of \eqref{eq:dual} with initial data $ \varphi_0^n$. By Proposition \ref{prop:extremal-reformulation},
we have $\underline{\varphi}_n \leq \underline{\varphi}$ for any integer $n$. We thus already know that $\underline{\varphi} \geq \sup_n \underline{\varphi}_n$ and it only remains to prove the other inequality. To do so, it suffices to show that $\sup_n \underline{\varphi}_n$ is a supersolution of \eqref{eq:dual} (with initial data $\varphi_0$). Indeed, by Proposition \ref{prop:extremal-reformulation}, we then get $\underline{\varphi} \leq \sup_n \underline{\varphi}
 _n$. It is at this stage that we need to use 
monotonicity. Recall that $n \mapsto \varphi_0^n(x)$ is nondecreasing for any $x$. By the comparison principle,  cf. Remark \ref{comp-extremal},  the same monotonicity holds for the  minimal  solutions which means that   $n \mapsto \underline{\varphi}_n(x,t)$  is nondecreasing for any fixed $x$ and $t$.  Since $ \underline{\varphi}_n$ is lower semicontinuous, we can argue as for \eqref{formula:extremal} and get  that 
$$
\sup_n \underline{\varphi}_n=\wliminf \underline{\varphi}_n,
$$
where the above relaxed limit is taken as $n \to \infty$. By stability, see Propositions \ref{prop:rsl} and \ref{prop:rsl-bis}, we deduce that $\wliminf \underline{\varphi}_n$ is a supersolution of \eqref{eq:dual-1} with initial data 
$$
\wliminf \underline{\varphi}_n(t=0) = \wliminf (\varphi_0^n)_\ast.
$$
But this initial data is precisely  
$$
\wliminf (\varphi_0^n)_\ast=\sup_n  (\varphi_0^n)_\ast=\varphi_0,
$$
again by similar arguments than for \eqref{formula:extremal}.  This completes the proof.
\end{proof}

\subsection{Representation formulas}
\label{app:representation}

Let us prove Propositions \ref{representation-deterministic} and \ref{rem-dyn}. These results are classical  in control  theory, but usually written for continuous or maximal solutions, see \cite{FlSo93,BaCa-Do97,GoSe10,GoSe11}. Here we give the proofs for minimal solutions. 

\begin{proof}[Proof of Proposition \ref{representation-deterministic}]
By the assumption that $a \equiv 0$, \eqref{eq:dual-1} is now
$$
\partial_t \varphi=\sup_{\xi \in \mathcal{E}} \{b(\xi) \cdot D \varphi\} =  \sup_{q \in \mathcal{C}} \{q \cdot D\varphi\},
$$
where $\mathcal{C}=\overline{\textup{co} \left\{\textup{Im} (b)\right\}}$ is compact. By control theory \cite{Bar94,BaCa-Do97} the viscosity solutions of \eqref{eq:dual} is given by
\begin{equation*}
\varphi(x,t)=\sup_{x+t \mathcal{C}} \varphi_0
\end{equation*}
if $\varphi_0$  is bounded and uniformly continuous.  In the general case, consider the infconvolution  \eqref{infconvol}. Recall that $(\varphi_0)_\veps$ is at least bounded and uniformly continuous,   and $(\varphi_0)_\veps \uparrow (\varphi_0)_\ast$ pointwise as $\veps \downarrow 0$. 
It follows that the solution of \eqref{eq:dual-1} with $(\varphi_0)_\veps$ as initial data is
$$
\varphi_\veps(x,t)=\sup_{x+t \mathcal{C}} (\varphi_0)_\veps.
$$
By Proposition \ref{stability-extremal}, the minimal solution of \eqref{eq:dual} is thus
\begin{equation*}
\underline{\varphi}(x,t)=\sup_{\veps >0} \varphi_\veps(x,t)=\sup_{\veps > 0} \sup_{x+t \mathcal{C}} (\varphi_0)_\veps=\sup_{x+t \mathcal{C}} \sup_{\veps > 0} \, (\varphi_0)_\veps=\sup_{x+t \mathcal{C}} (\varphi_0)_\ast. 
\end{equation*}
 Rigorously speaking, Proposition \ref{stability-extremal} implies that this is the minimal solution with initial data $(\varphi_0)_\ast$, but it coincides with the minimal solution associated to $\varphi_0$ by Proposition \ref{prop:extremal-reformulation}.
\end{proof}

\begin{proof}[Proof of Proposition \ref{rem-dyn}]
Equation \eqref{eq:dual-1} is given by 
$$
\partial_t \varphi=\sup_{\xi \in \mathcal{E}} \left\{b(\xi) \cdot D\varphi+\tr \left(\sigma^a(\xi) (\sigma^a)^\T(\xi) D^2 \varphi \right) \right\},
$$
where $\mathcal{E}$ is compact and the coefficients $b$ and $\sigma^a$ are continuous by \eqref{ass:regular}.
By stochastic control theory \cite{FlSo93}, the viscosity solution of \eqref{eq:dual} is given by
$$
\varphi(x,t)=\sup_{\boldsymbol{\xi}_\cdot \in\Bigxi} \mathbb{E} \left\{\varphi_0(\boldsymbol{X}_t^x)\right\}
$$ 
 if  $\varphi_0$  is bounded and uniformly continuous, 
where $\Bigxi$ and $\boldsymbol{X}_s^x$ are defined in Proposition \ref{rem-dyn}.
Let us now repeat the argument of the proof of Proposition \ref{representation-deterministic} considering the infconvolution $(\varphi_0)_\veps$ and the corresponding solution of \eqref{eq:dual-1}
$$
\varphi_\veps(x,t)=\sup_{\boldsymbol{\xi}_\cdot\in\Bigxi}
\mathbb{E} \left\{(\varphi_0)_\veps (\boldsymbol{X}_t^x)\right\}.
$$
We find that the minimal solution of \eqref{eq:dual} is 
$$
\underline{\varphi}(x,t)=\sup_{\veps>0}\varphi_\veps(x,t)=\sup_{\boldsymbol{\xi}_\cdot\in\Bigxi} 
\, \, \sup_{\veps>0} \,  \mathbb{E} \left\{ (\varphi_0)_\veps(\boldsymbol{X}_t^x)\right\}.
$$
Since $(\varphi_0)_\veps\uparrow(\varphi_0)_*$ as $\veps\downarrow0$, we conclude the proof using the monotone convergence theorem:
\begin{align*}
&\sup_{\veps>0} \,  \mathbb{E} \left\{ (\varphi_0)_\veps(\boldsymbol{X}_t^x)\right\}\\ 
&=\lim_{\veps \downarrow0}{\uparrow} \,  \mathbb{E} \left\{(\varphi_0)_\veps(\boldsymbol{X}_t^x)\right\}=
\mathbb{E} \Big\{\lim_{\veps \downarrow0}{\uparrow} (\varphi_0)_\veps(\boldsymbol{X}_t^x) \Big\}=\mathbb{E} \left\{(\varphi_0)_\ast(\boldsymbol{X}_t^x) \right\}.\qedhere
\end{align*}
\end{proof}

\section{Complementary proofs for entropy solutions}\label{app:entropy}

For completeness, 
we recall the proof of  Theorem \ref{thm:existenceuniquenessentropysoln} which is Theorem 1.1 in \cite{Fripr}  under \eqref{fassumption}.  
We will take the opportunity to give details, but we will not perform the doubling of variables to show Lemma \ref{dualequation} for which we will refer to \cite{BeKa04}.  

Recall that \cite{ChPe03,BeKa04}  proved the well-posedness of $L^1$ kinetic or renormalized solutions
which are
equivalent to entropy solutions in $L^1 \cap L^\infty$. The definition of entropy solutions in $L^1 \cap L^\infty$ uses the energy estimate (2.8) of \cite{ChPe03},
\begin{equation*}
\iint_{\R^d \times \R^+} \sum_{k=1}^K\left(\sum_{i=1}^d\partial_{x_i}\zeta_{ik}(u)\right)^2 \dd x \dd t
\leq \frac{1}{2} \|u_0\|_{L^2}<\infty \quad \mbox{if $u_0 \in L^1 \cap L^\infty$,}
\end{equation*}
where 
$
\zeta_{ik}(u)=\int_0^u\sigma^{\scriptscriptstyle A}_{ik}(\xi)\dd \xi.
$ 
As a consequence ``$L^2$'' was used e.g. in \cite[Definition 2.2]{BeKa04} instead of ``$L^2_\textup{loc}$'' in Definition \ref{defentropy}. But we have the following result: 

\begin{lemma}[Local energy estimate]\label{finiteenergy}
Assume \eqref{fassumption}, $u_0\in L^\infty(\R^d)$, $0 \leq \phi \in C^\infty_c(\R^d)$, and $T \geq 0$. If $u$ is an entropy solution of \eqref{E}  in the sense of Definition \ref{defentropy}  and
$$
\|u_0\|_{L^\infty}+\|u\|_{L^\infty}+\|\phi\|_{W^{2,1}} \leq M,
$$ 
then there is a constant $C$ only depending on $T$, $M$, $F$ and $A$ such that
\begin{equation*}
\iint_{\R^d \times (0,T)} \sum_{k=1}^K \left(\sum_{i=1}^d\partial_{x_i}\zeta_{ik} (u(x,t)) \right)^2\phi(x) \dd x\dd t \leq C.
\end{equation*}
\end{lemma}

\begin{proof}
We use Definition \ref{defentropy}\eqref{item:entropy} with
the entropy $\eta(u)=|u|^2$ and the corresponding fluxes 
$$
q(u)=2\int_0^u\xi F'(\xi)\dd \xi \quad
\mbox{and} \quad r(u)=2\int_0^u\xi A(\xi)\dd \xi.
$$
We also take a test function $\phi(x) \mathbf{1}_{[0,T]}(t)$ where
$
0 \leq \phi \in C_c^\infty(\R^d)$. It is not smooth in time but a standard approximation argument shows that it can be used in Definition \ref{defentropy}\eqref{item:entropy} if we add also a final value term at $t=T$. Here we need the $L^1_\textup{loc}$ continuity in time of entropy solutions. The result is
\begin{multline*}
\overbrace{\int_{\R^d}u^2(x,T) \phi(x)\dd x}^{\geq 0}+2 \iint_{\R^d \times (0,T)} \sum_{k=1}^K\left(\sum_{i=1}^d\partial_{x_i}\zeta_{ik}(u)\right)^2 \phi \dd x \dd t\\
\leq \int_{\R^d}u_0^2(x) \phi(x)\dd x+\iint_{\R^d \times (0,T)} \left(\sum_{i=1}^dq_i(u)\partial_{x_i}\phi+\sum_{i,j=1}^{d}r_{ij}(u)\partial_{x_ix_j}^2\phi \right) \dif x\dd t.
\end{multline*}
By assumption $\|u_0\|_{L^\infty}+\|u\|_{L^\infty}+\|\phi\|_{W^{2,1}} \leq M$, so it follows that
\begin{equation*}
\begin{cases}
\|q(u)\|_{L^\infty(\R^d \times \R^+,\R^d)} \leq 2 M^2 \esssup_{-M \leq \xi \leq M} |F'(\xi)|, \mbox{ and } \\
\|r(u)\|_{L^\infty(\R^d \times \R^+,\R^{d \times d})} \leq 2 M^2 \esssup_{-M \leq \xi \leq M} |A(\xi)|.
\end{cases}
\end{equation*}
With all these estimates, the conclusion readily follows.
\end{proof}

Let us  now  give precise references on how to show the Kato inequality.  

\begin{proof}[Sketch of the proof of Lemma \ref{dualequation}]
Copy the proof of Theorem 3.1 of \cite{BeKa04} with $l=\infty$ and zero renormalization measures  $\mu_l^u \equiv 0 \equiv  \mu_l^v$.  
With the aid of the previous local energy estimate, check
that every 
computation holds until (3.19)  --  even if $u$ and $v$ satisfy \eqref{item:energy}--\eqref{item:chain} of Definition \ref{defentropy} with $L^2_\textup{loc}$ and not $L^2$ as in \cite{BeKa04}.
This gives \eqref{dualeq} with $\phi \in C_c^\infty(\R^d \times (0,\infty))$. Use an approximation argument for  $\phi(x,t) \mathbf{1}_{[0,T]}(t)$  and the continuity in time with values in  $L^1_\textup{loc}$  to get initial and final terms.
\end{proof}

To show the uniqueness of entropy solutions, it suffices to find a  {\em good}  $\phi$
in \eqref{dualeq},
e.g.  an exponential as in \cite{ChDi01,Fripr}. This gives the result below.

\begin{lemma}
\label{standarduniqueness}
Assume \eqref{fassumption} and $u,v$ are $L^\infty$ entropy solutions of \eqref{E} with initial data $u_0, v_0\in L^\infty(\R^d)$. Then for any $t \geq 0$  and $m<M$ such that $u$ and $v$ take their values in $[m,M]$,  
$$
\int_{\R^d} |u-v|(x,t)\e^{-|x|}\dd x\leq\e^{(L_F+L_A) t}\int_{\R^d} |u_0-v_0|(x)\e^{-|x|}\dd x,
$$
where  $L_F
 = 
\esssup_{[m,M]} |F'|$ and $L_A
 = 
\esssup_{[m,M]} \tr (A)$. 
\end{lemma}
\begin{remark}
 By the maximum principle, the result remains true for any $[m,M]$ containing the values $u_0$ and $v_0$. But  at this stage of this appendix, this principle is only known in $L^1 \cap L^\infty$ (or $L^1$) by \cite{ChPe03,BeKa04} and it will follow later in $L^\infty$.
\end{remark}
\begin{proof}[Sketch of the proof]
The proof is inspired by \cite{ChDi01,Fripr}. 
Consider
$$
\phi_\varepsilon(x,t):=\e^{(L_F+L_A) (T-t)-\sqrt{\varepsilon^2+|x|^2}},
$$
for some arbitrary $\varepsilon >0$, and check that
\begin{align*}
& |u-v| \partial_t \phi_\varepsilon+\sum_{i=1}^d q_i(u,v)\partial_{x_i}\phi_\varepsilon+\sum_{i,j=1}^d r_{ij}(u,v)\partial_{x_ix_j}^2\phi_\varepsilon \\
& \leq |u-v| \left\{\partial_t \phi_\varepsilon+L_F |D \phi_\varepsilon|+ L_A  \sup_{\lambda \in \textup{Sp}(D^2 \phi_\varepsilon)} \lambda^+\right\} \leq 0
\end{align*}\normalcolor
by the Ky Fan inequality \eqref{kyfan}. Then by the Kato inequality \eqref{dualeq} with $\phi_\varepsilon$, 
$$
\int_{\R^d} |u-v|(x,T)\e^{-\sqrt{\varepsilon^2+|x|^2}}\dd x\leq\e^{(L_F+L_A) T }\int_{\R^d} |u_0-v_0|(x)\e^{-\sqrt{\varepsilon^2+|x|^2}}\dd x
$$
and the result follows in the limit $\varepsilon \to 0^+$.
\end{proof}

\begin{proof}[Proof of Theorem \ref{thm:existenceuniquenessentropysoln}]
By Lemma \ref{standarduniqueness}, it remains to show the existence.
The proof is inspired by \cite{ChPe03,BeKa04}. 
Given $u_0 \in L^\infty(\R^d)$, take $(u_{0}^n)_n$ in $L^1 \cap L^\infty(\R^d)$  such that
\begin{equation}\label{propsofinitialdataseq}
 -\esssup u_{0}^-  
\leq u_0^n \leq 
 \esssup u_0^+ 
\quad \text{and} \quad u_{0}^n \to u_0 \quad \text{in $L_\textup{loc}^1(\R^d)$}.
\end{equation}
Let $u_n$ be the entropy solution of \eqref{E} with initial data $u_{0}^n$. By the maximum principle (in $L^1 \cap L^\infty$), we know that
\begin{equation}\label{propsofinitialdataseq-bis}
 -\esssup u_{0}^-  
\leq u_n \leq 
 \esssup u_0^+. 
\end{equation}
Moreover, by Lemma \ref{standarduniqueness}, we have for any $R \geq 0$, $T \geq 0$, and integers $n$, $m$,
\begin{equation*}
\begin{split}
&\|u_m-u_n\|_{C([0,T];L^1(\{|x| < R\}))}\\
&=\sup_{t\in[0,T]}\int_{|x|<R}|u_m(x,t)-u_n(x,t)|\dd x\\
&\leq \e^{R}\sup_{t\in[0,T]}\int_{\R^d}|u_m(x,t)-u_n(x,t)|\e^{-|x|}\dd x\\
&\leq\e^{R} \e^{(L_F+L_A) T}\int_{\R^d}|u_{0}^m(x)-u_{0}^n(x)|\e^{-|x|}\dd x,
\end{split}
\end{equation*}
where the latter integral tends to zero as $n,m \to \infty$ by \eqref{propsofinitialdataseq}. Hence there exists some $u \in L^\infty(\R^d \times \R^+) \cap C(\R^+;L^1_{\textup{loc}}(\R^d))$ such that
\begin{equation}\label{strongconvofseqentropy}
\lim_{n \to \infty} u_n = u \quad \text{in } C([0,T];L_\textup{loc}^1(\R^d)), \quad  \forall T \geq 0.
\end{equation}
It remains to show that $u$ is an entropy solution with initial data $u_0$. 

 We have to derive 
the $L^2_{\textup{loc}}$ energy estimate of Definition \ref{defentropy}\eqref{item:energy}, and check that it is enough to pass to the limit in the equation as in \cite{ChPe03,BeKa04}. 
By Lemma \ref{finiteenergy} and the $L^\infty$ bounds in 
\eqref{propsofinitialdataseq-bis}, 
the sequence 
$$
\left\{\sum_{i=1}^d\partial_{x_i}\zeta_{ik}(u_n)\right\} \subset L^2(\R^d \times \R^+)
$$ 
is uniformly bounded in $L^2( \mathcal{K})$, for any  $k=1,\ldots,K$, and compact  $\mathcal{K}\subset\R^d \times \R^+$. It then weakly converges in $L^2(\mathcal{K})$ to $\sum_{i=1}^d \partial_{x_i}\zeta_{ik}(u)$. We can identify the limit because $\sum_{i=1}^d\partial_{x_i}\zeta_{ik}(u_n)$ also converges to $\sum_{i=1}^d \partial_{x_i}\zeta_{ik}(u)$ in the distribution sense. Indeed
$$
\zeta_{ik}(\cdot)=\int_0^{\cdot} \sigma^\A_{ik} (\xi) \dd \xi 
$$
is locally Lipschitz continuous since $\sigma^\A_{ik}(\cdot)$ is locally bounded, and \eqref{propsofinitialdataseq-bis} and \eqref{strongconvofseqentropy} imply that $\zeta_{ik}(u_n) \to \zeta_{ik}(u)$ in $C([0,T];L^1_{\textup{loc}}(\R^d))$ for all $T \geq 0$. And as claimed previously, all corresponding derivatives necessarily converge in the distribution sense.  The proof of 
part
\eqref{item:energy} in Definition \ref{defentropy} is complete. Moreover we have found that
\begin{equation*}
\sum_{i=1}^d\partial_{x_i}\zeta_{ik}(u_n)\rightharpoonup\sum_{i=1}^d\partial_{x_i}\zeta_{ik}(u)\quad\text{in } L^2(\mathcal{K}),
\end{equation*}
for any $k=1,\dots,K$ and compact $\mathcal{K} \subset \R^d \times \R^+$. 

To show the chain rule in part \eqref{item:chain} of Definition \ref{defentropy}, we start from the chain rule for $u_n$,
\begin{equation}
\label{chainruleu_n}
\sum_{i=1}^d\partial_{x_i}\zeta_{ik}^\chain(u_n)=\chain(u_n)\sum_{i=1}^d\partial_{x_i}\zeta_{ik}(u_n) \in L^2(\R^d \times \R^+),
\end{equation}
valid for any $\chain\in C(\R)$, $k=1,\ldots,K$, and integer $n$. Recall also that
$$
\zeta_{ik}^\chain(u_n)=\int_0^{u_n} \sigma^\A_{ik} (\xi) \beta (\xi) \dd \xi.
$$
By the previous convergence results and bounds, the right-hand side of \eqref{chainruleu_n} 
converges weakly in $L^2(\mathcal{K})$ to 
$\chain(u)\sum_{i=1}^d\partial_{x_i}\zeta_{ik}(u)$.  
We can argue as before to show that the left-hand side
converges weakly in $L^2(\mathcal{K})$ to $\sum_{i=1}^d\partial_{x_i}\zeta_{ik}^\chain(u)$. We thus get 
part
\eqref{item:chain} of Definition \ref{defentropy} in the limit. Moreover,
\begin{equation}\label{weakenergyconv}
\sum_{i=1}^d\partial_{x_i} \zeta_{ik}^\chain(u_n)\rightharpoonup\sum_{i=1}^d\partial_{x_i}\zeta^\chain_{ik}(u)\quad\text{in } L^2(\mathcal{K}),
\end{equation}
for any $\chain \in C(\R)$, $k=1,\dots,K$, and compact $\mathcal{K} \subset \R^d \times \R^+$. 

Now, it remains to  prove part  \eqref{item:entropy} of Definition \ref{defentropy}. The only difference with \cite{ChPe03,BeKa04} is that the previous convergences hold locally in $L^2$ and not globally. But since we use test functions, the reasoning is the same. Let us recall it for completeness. We focus on the quadratic term. Take $\beta=\sqrt{\eta''}$ and apply the chain rule Definition \ref{defentropy}\eqref{item:chain}, 
\begin{equation*}
\begin{split}
&\iint_{\R^d \times \R^+}\eta''(u_n)\sum_{k=1}^K\left(\sum_{i=1}^d\partial_{x_i}\zeta_{ik}(u_n)\right)^2\phi\dd x \dd t\\
&=\iint_{\R^d \times \R^+}\eta''(u_n)\sum_{k=1}^K\left(\sum_{i=1}^d\partial_{x_i}\zeta_{ik}(u_n) \right) \left(\sum_{j=1}^d\partial_{x_j}\zeta_{jk}(u_n) \right) \phi\dd x \dd t\\
&=\iint_{\R^d \times \R^+}\sum_{k=1}^K\left(\sum_{i=1}^d\partial_{x_i}\zeta_{ik}^{\sqrt{\eta''}}(u_n)\right)\left(\sum_{j=1}^d\partial_{x_j}\zeta_{jk}^{\sqrt{\eta''}}(u_n)\right)\phi\dd x \dd t\\
&=\iint_{\R^d \times \R^+}\sum_{k=1}^K\left(\sum_{i=1}^d\partial_{x_i}\zeta_{ik}^{\sqrt{\eta''}}(u_n)\sqrt{\phi}\right)^2\dd x \dd t\\
& = \sum_{k=1}^K \left\|\sum_{i=1}^d\partial_{x_i}\zeta_{ik}^{\sqrt{\eta''}}(u_n)\sqrt{\phi}\right\|_{L^2(\R^d \times \R^+)}^2.
\end{split}
\end{equation*}
But, by \eqref{weakenergyconv}, we have for any $k=1,\dots,K$,
$$
\sum_{i=1}^d\partial_{x_i} \zeta_{ik}^{\sqrt{\eta''}}(u_n) \sqrt{\phi}\rightharpoonup\sum_{i=1}^d\partial_{x_i}\zeta_{ik}^{\sqrt{\eta''}}(u) \sqrt{\phi} \quad \mbox{in } L^2(\R^d \times \R^+).
$$
It follows that
$$
\left\|\sum_{i=1}^d\partial_{x_i}\zeta_{ik}^{\sqrt{\eta''}}(u)\sqrt{\phi}\right\|_{L^2(\R^d \times \R^+)}\leq \liminf_{n\to\infty}\left\|\sum_{i=1}^d\partial_{x_i}\zeta_{ik}^{\sqrt{\eta''}}(u_n)\sqrt{\phi}\right\|_{L^2(\R^d \times \R^+)},
$$
that is
\begin{equation*}
\begin{split}
&\liminf_{n\to\infty}\iint_{\R^d \times \R^+}\eta''(u_n)\sum_{k=1}^K\left(\sum_{i=1}^d\partial_{x_i}\zeta_{ik}(u_n)\right)^2\phi\dd x \dd t\\
& \geq \sum_{k=1}^K \left\|\sum_{i=1}^d\partial_{x_i}\zeta_{ik}^{\sqrt{\eta''}}(u)\sqrt{\phi}\right\|_{L^2(\R^d \times \R^+)}^2\\
& = \iint_{\R^d \times \R^+}\eta''(u)\sum_{k=1}^K\left(\sum_{i=1}^d\partial_{x_i}\zeta_{ik}(u)\right)^2\phi\dd x \dd t,
\end{split}
\end{equation*}
where similar chain rule computations have been used for $u$. This is enough to pass to the limit in the entropy inequalities of Definition \ref{defentropy}\eqref{item:chain} and the proof is complete.
\end{proof}

As a byproduct of the previous proof, we get the lemma below.

\begin{lemma}\label{lem:max-comp}
Assume \eqref{fassumption}, $u_0 \in L^\infty(\R^d)$, and $u$ is the entropy solution of \eqref{E}. Then $\essinf u_0 \leq u \leq \esssup u_0$. Moreover, if $v$ is the entropy solution with initial data $v_0$, then $u_0 \geq v_0$ implies $u \geq v$.
\end{lemma}

\begin{proof}
For the comparison principle, define $u_0^n(x):=u_0(x) \mathbf{1}_{|x|<n}$ and $v_0^n$ similarly. As previously, the associated entropy solutions $u_n$ and $v_n$ respectively converge towards $u$ and $v$ in $C([0,T];L^1_{\rm loc}(\R^d))$, $T \geq 0$, and thus  a.e.  up to taking a (common) subsequence. If $u_0 \geq v_0$, then $u_0^n \geq v_0^n$ for all $n$, so $u_n \geq v_n$ by the comparison principle in $L^1 \cap L^\infty$, and $u \geq v$ at the limit. For the maximum principle, apply the comparison principle to $v_0:=\essinf u_0$ and $\esssup u_0$ successively. 
\end{proof}

\section{Measurable weights and viscosity supersolutions}\label{app-comp-referee}

Let us provide for completeness a version of Theorem \ref{thm:dual} for measurable and essentially bounded weights $\varphi:\R^d \times \R^+ \to \R$. The result will involve a version of $\varphi_\#$ from \eqref{precise-representative} in both space and time. It is defined as
\begin{equation}
\label{precise-representative-referee}
\varphi_{\#\#}(x,t):= \liminf_{\substack{r \to 0^+ \\ y \to x \\ \R^+ \ni s \to t}} \frac{1}{r \, \textup{meas}(B_r(y))} \iint_{B_r(y)\times(s,s+r)} \varphi(z,\tau ) \dd z \dd \tau.
\end{equation} 
Notably $\varphi_{\#\#} \in BLSC$ with $\varphi_{\#\#} \leq \varphi$ a.e., but  we may not have $\varphi_{\#\#}=\varphi$ a.e. when $\varphi \notin BLSC$; cf. Remark \ref{rem:counter-example}.

\begin{theorem}[Measurable weights and supersolutions]\label{thm:measurable} Assume \eqref{fassumption}, $m<M$, and $\varphi:\R^d \times \R^+ \to \R$ is measurable, nonnegative, essentially bounded, and such that 
\begin{equation}\label{optimal-dual-measurable}
\int_{\R^d} |u-v|(x,t) \varphi(x,s) \dd x \leq \int_{\R^d} |u_0-v_0|(x) \varphi(x,t+s)\dd x \quad \text{a.e. $t, s \geq 0$},
\end{equation}
for any $u_0$ and $v_0$ in $L^\infty(\R^d,[m,M])$ with respective associated entropy solutions $u$ and $v$ of \eqref{E}. Then $\varphi_{\#\#}$ in \eqref{precise-representative-referee}
is a viscosity supersolution of \eqref{eq:dual-true-1}.
\end{theorem}

\begin{remark}\label{rem:counter-example}
The reciprocal assertion may fail. A one dimensional example is $\varphi(x,t)=\mathbf{1}_{E}(x)$ with a fat Cantor set $E$ (a closed nowhere dense set of positive measure). Indeed $\varphi_{\#\#} \equiv 0$\footnote{Use that $\varphi_{\#\#}$ is $LSC$, nonnegative, and equals zero in the dense open set $(\R \setminus E) \times \R^+$.} is always a solution of \eqref{eq:dual-true-1}, but chosing \eqref{E} as the heat equation $\partial_t u=\partial_{xx}^2 u$, $u_0=\mathbf{1}_{\R \setminus E}$, and $v_0 \equiv 0$, we cannot have \eqref{optimal-dual-measurable} because the right-hand side is zero and the left-hand side is positive.
\end{remark}

\begin{remark}[The optimal measurable weight is $\underline{\varphi}$]
Given in addition $0 \leq \varphi_0 \in BLSC(\R^d)$ such that
\begin{equation}\label{initial-data-referee}
\varphi_0(x) \leq \varphi_{\#\#}(x,t=0) \quad \text{for all $x \in \R^d$},
\end{equation}
we have
$$
\int_{\R^d} |u-v|(x,t) \varphi_0(x) \dd x \leq \int_{\R^d} |u_0-v_0|(x) \varphi(x,t)\dd x \quad \text{a.e. $t \geq 0$}.
$$
This is \eqref{nonlinear-dual} with the merely measurable weight $\varphi$. Notably the minimal viscosity solution $\underline{\varphi}$ of \eqref{eq:dual-true} remains optimal within this class of weights satisfying \eqref{optimal-dual-measurable} and \eqref{initial-data-referee}, because $\underline{\varphi} \leq \varphi_{\#\#} \leq \varphi$ where the last inequality holds a.e.
\end{remark}

\begin{remark}
\begin{enumerate}[\rm (a)]
\item Going back to $\varphi \in BLSC$ satisfying Theorem \ref{thm:dual}\eqref{nonlinear-dual-optimal}, and applying Theorems \ref{thm:dual} and \ref{thm:measurable}, we get two viscosity supersolutions $\varphi_\#$ and $\varphi_{\#\#}$ of \eqref{eq:dual-true-1}. This is however coherent because they actually represent the same supersolution. Indeed $\varphi_{\#\#}$ is nothing else than $\varphi_\flat$ in \eqref{def-varphi-flat}, and we have seen during the proof of \eqref{nonlinear-dual-optimal}$\implies$\eqref{item:dual}  that $(\varphi_\#)_\ast=\varphi_\flat=\varphi_{\#\#}$ pointwise.
\item For general $\varphi \in BLSC$, $(\varphi_\#)_\ast \leq  \varphi_{\#\#}$ pointwise by Lemma \ref{lem:measurability} but the reverse inequality may fail. An example is $\varphi(x,t)=\mathbf{1}_{t \neq t_0}$ with some fixed $t_0$, which gives $(\varphi_\#)_\ast=\varphi$  and $\varphi_{\#\#} \equiv 1$.
\end{enumerate}
\end{remark}

We actually already proved the above theorem since $\varphi_{\#\#}=\varphi_\flat$ from \eqref{def-varphi-flat}. But let us give details for the reader's convenience.

\begin{proof}[Proof of Theorem \ref{thm:measurable}]
By ``a.e.'' in \eqref{optimal-dual-measurable}, we assume having a null set $N \subset \R^+$ such that \eqref{optimal-dual-measurable} holds for all $t,s \geq 0$ such that $ s \notin N$ and $t+s \notin N$.
Fix $r>0$ and define
$$
\varphi_r(x,t):=\frac{1}{r \, \textup{meas}(B_r(x))}  \iint_{B_r(x)\times(t,t+r)} \varphi(y,s) \dd y \dd s.
$$ 
As for \eqref{last-referee}, it is easy to deduce from \eqref{optimal-dual-measurable} that
$$
\int_{\R^d} |u-v|(x,t) \varphi(x-y,s-\tau) \dd x \leq \int_{\R^d} |u_0-v_0|(x) \varphi(x-y,t+s-\tau)\dd x,
$$
for all $y \in \R^d$, $t \geq 0$, and $s-\tau \geq 0$, such that $s-\tau \notin N$ and $t+s-\tau \notin N$.
Fix $t,s \geq 0$, multiply by $\frac{1}{r \, \textup{meas}(B_r(x))}$ and integrate over $(y,\tau) \in B_r(0) \times (-r,0)$, which we can do excepted for $\tau \in (s-N) \cup (t+s-N)$. But the latter set is a null set, and this shows that $\varphi_r$ satisfies \eqref{optimal-dual-measurable} for all $t,s \geq 0$. Since moreover $\varphi_r$ is continuous in $(x,t)$, it is a viscosity supersolution of \eqref{eq:dual-true-1} by Theorem \ref{thm:dual} and so is $\varphi_{\#\#}=\wliminf \varphi_r$ as $r \to 0^+$.
\end{proof}

\section{Nonlinear to linear semigroups}\label{app-referee}

In this section we give a sample result on how we from nonlinear duality can recover standard duality notions in the linear case. It contains the discussion and results mentioned in Remark \ref{noq}(c) and the notation and setting is taken from Section \ref{sec:unif-int}.
First note that $X=C_b \cap L^\infty_\INT$ was a natural space for the weight semigroup  $G_t$, but other $X$ could be more appropriate if we consider  
other semigroups than $S_t$. Here are some reasonable assumptions which we will need:
\begin{equation}\label{X-referee}
\text{$X \neq \emptyset$ is a Banach space continuously embedded and dense in $L^1$,}
\end{equation}
such that
\begin{equation}\label{lattice-referee}
\forall \varphi_0 \in L^1, \forall \psi_0 \in X, \quad |\varphi_0| \leq |\psi_0| \Rightarrow \big[\varphi_0 \in X \text{ and } \|\varphi_0\|_X \leq \|\psi_0\|_X\big],
\end{equation}
(i.e. $X$ is a Banach lattice) and for any mollifier $\rho_\nu$ 
(cf. \eqref{space-approx}) and  $\varphi_0 \in X$,
\begin{equation}\label{mollifier-referee}
X \ni \rho_\nu \ast \varphi_0 \to \varphi_0 \text{ strongly in $X$ with } \|\rho_\nu \ast \varphi_0\|_X \leq \|\rho_\nu\|_X \|\varphi_0\|_{L^1}.
\end{equation}
Note that hereafter $L^1=L^1(\R^d)$ etc.

\begin{proposition}[Relation with standard duality]
Take a weakly-$\star$ continuous semigroup $T_t$ of weakly-$\star$ continuous linear operators on $L^\infty=(L^1)^\star$ such that each $T_t$ is positive and commutes with translations.\footnote{That is $T \varphi_0 \geq 0$ if $\varphi_0 \geq 0$, and $T (\varphi_0(\cdot+h))=(T \varphi_0)(\cdot+h)$ for all $\varphi_0 \in L^\infty$ and $h \in \R^d$.} Let $(T_t)_\star$ be its predual semigroup on $L^1$ defined by
\begin{equation}\label{standard-dual-referee}
\int_{\R^d} \varphi_0 T_t u_0 \dd x=\int_{\R^d} u_0 (T_t)_\star\varphi_0 \dd x,  \quad  \forall u_0 \in L^\infty, \forall \varphi_0 \in L^1, \forall t\geq 0.
\end{equation}
Assume also that there exist $X$ satisfying \eqref{X-referee}--\eqref{lattice-referee}--\eqref{mollifier-referee}, a strongly continuous semigroup $H_t$ of continuous operators on $X^+:=\{\varphi_0 \in X:\varphi_0 \geq 0\}$ satisfying
\begin{equation}\label{etoile-bis-referee}
\int_{\R^d} |T_tu_0-T_tv_0| \varphi_0 \dd x \leq \int_{\R^d} |u_0-v_0| H_t \varphi_0 \dd x, \quad \forall  u_0,v_0 \in L^\infty, \forall \varphi_0 \in X^+, \forall t \geq 0,
\end{equation}
and that $H_t$ is the minimal such semigroup. Then $(T_t)_\star$ is necessarily the unique extension of $H_t$ from $X^+$ onto $L^1$ as a semigroup of bounded linear operators.
\end{proposition}

\begin{remark}
For a general duality theory for linear semigroups, see \cite{vNe92}. Let us recall that \eqref{standard-dual-referee} defines a strongly continuous semigroup $(T_t)_\star$ of bounded linear operators on $L^1$.\footnote{Indeed, given $\varphi_0 \in L^1$, $u_0 \in L^\infty \mapsto \int \varphi_0 T_t u_0$ is weakly-$\star$ continuous thus corresponding to a unique element $(T_t)_\star \varphi_0 \in L^1 \subset (L^\infty)^\star$. This operator $(T_t)_\star$ is bounded in $L^1$ since $T_t$ is weakly-$\star$ continuous in $L^\infty$, thus bounded. The semigroup $(T_t)_\star$ is weakly continuous thus strongly continuous.}
The semigroup $H_t$ would be the new predual defined as following Remark \ref{noq}\eqref{item-referee}, which would thus coincide with $(T_t)_\star$ in the linear case.
\end{remark}

\begin{proof}
Take $u_0 \geq 0$, $v_0 \equiv 0$ and $\varphi_0 \geq 0$ in \eqref{etoile-bis-referee}, to get
$$
\int_{\R^d} u_0 (T_t)_\star \varphi_0 \dd x =\int_{\R^d} 
\varphi_0 T_t u_0 \dd x \leq \int_{\R^d} u_0 H_t
\varphi_0 \dd x.
$$ 
This shows that
\begin{equation}\label{first-ineq-referee}
(T_t)_\star \leq H_t \quad \mbox{ on $X^+$}.
\end{equation}
To continue, we claim that $(T_t)_\star$ is a strongly continuous semigroup of continuous operators on $X^+$ satisfying \eqref{etoile-bis-referee}. Let us verify this claim. Let us prove that $(T_t)_\star$ satisfies \eqref{etoile-bis-referee}, as $H_t$ does. Since $T_t \geq 0$ and is linear, 
$$
|T_t u_0-T_t v_0|=|T_t (u_0-v_0)^+-T_t (u_0-v_0)^-| \leq T_t (u_0-v_0)^++T_t (u_0-v_0)^-
$$
for any $u_0$ and $v_0$ in $L^\infty$. Hence
\begin{multline*}
\int_{\R^d} |T_t u_0-T_t v_0| \varphi_0 \dd x \leq \int_{\R^d} \left(T_t (u_0-v_0)^++T_t (u_0-v_0)^-\right) \varphi_0 \dd x\\
=\int_{\R^d} \left((u_0-v_0)^++(u_0-v_0)^-\right)  (T_t)_\star  \varphi_0 \dd x=\int_{\R^d} |u_0-v_0|   (T_t)_\star  \varphi_0 \dd x,
\end{multline*}
for any $\varphi_0 \in X^+$. To show next
that $(T_t)_\star$ is bounded for $\|\cdot\|_X$, we use that $(T_t)_\star \geq 0$,  the previous bound \eqref{first-ineq-referee}, the assumption \eqref{lattice-referee}, and the continuity of $H_t$ for this norm. For the time continuity of $(T_t)_\star$, we regularize any $\varphi_0 \in X$ by convolution thanks to \eqref{mollifier-referee}. Take $\varphi_0^\nu := \rho_\nu \ast \varphi_0 \to \varphi_0$ in $X$ as $\nu \to 0^+$, and note that
\begin{equation}\label{tech-referee}
\|\varphi_0-(T_t)_\star \varphi_0\|_{X} \leq \|\varphi_0-\varphi_0^\nu\|_{X}
+ \| \varphi_0^\nu-(T_t)_\star \varphi_0^\nu\|_{X}+ \underbrace{\|(T_t)_\star (\varphi_0^\nu - \varphi_0)\|_{X}}_{\leq \sum_{\pm} \|H_t (\varphi_0^\nu -\varphi_0)^\pm\|_{X} \text{ by \eqref{first-ineq-referee}}}.
\end{equation}
Note also that 
$
(T_t)_\star (\varphi_0 \ast \rho_\nu)=\rho_\nu \ast (T_t)_\star \varphi_0
$
since $(T_t)_\star:L^1 \to L^1$ is linear, bounded, and commutes with translations. Hence
$$
\|\varphi_0^\nu-(T_t)_\star \varphi_0^\nu\|_{X}=\left\|\rho_\nu \ast \left(\varphi_0-(T_t)_\star \varphi_0\right)\right\|_{X} \leq \|\rho_\nu\|_{X} \| \varphi_0-(T_t)_\star \varphi_0\|_{L^1}
$$
by \eqref{mollifier-referee}, and letting $t \to 0^+$ before $\nu \to 0^+$ in \eqref{tech-referee} implies that 
$$
\lim_{t \to 0^+} \|\varphi_0-(T_t)_\star \varphi_0\|_{X} =0
$$ 
by the (time) strong continuity of $(T_t)_\star$ on $L^1$. This proves our claim, and we infer that $(T_t)_\star = H_t$ on $X^+$. The result follows by density.
\end{proof}

\end{document}